\documentclass[11pt]{amsart}
\usepackage{amsfonts}
\usepackage{amsmath}
\usepackage{amssymb}
\usepackage{mathrsfs}
\usepackage{geometry}
\usepackage{graphicx}
\usepackage{subfigure}
\usepackage{hyperref}
\usepackage{microtype}
\usepackage{enumerate}
\usepackage{epsfig}
\usepackage{cancel}
\usepackage{color}
\usepackage{cite}
\usepackage{makecell}
\usepackage{multirow}
\usepackage{footmisc}
\usepackage{wrapfig}
\usepackage{array}
\usepackage{titletoc}

\makeatletter
\def\l@subsection{\@tocline{2}{0pt}{2.5pc}{5pc}{}}
\renewcommand\tocchapter[3]{%
  \indentlabel{\@ifnotempty{#2}{\ignorespaces#2.\quad}}#3%
}
\newcommand\@dotsep{4.5}
\def\@tocline#1#2#3#4#5#6#7{\relax
  \ifnum #1>\c@tocdepth 
  \else
    \par \addpenalty\@secpenalty\addvspace{#2}%
    \begingroup \hyphenpenalty\@M
    \@ifempty{#4}{%
      \@tempdima\csname r@tocindent\number#1\endcsname\relax
    }{%
      \@tempdima#4\relax
    }%
    \parindent\z@ \leftskip#3\relax \advance\leftskip\@tempdima\relax
    \rightskip\@pnumwidth plus1em \parfillskip-\@pnumwidth
    #5\leavevmode\hskip-\@tempdima{#6}\nobreak
    \leaders\hbox{$\m@th\mkern \@dotsep mu\hbox{.}\mkern \@dotsep mu$}\hfill
    \nobreak
    \hbox to\@pnumwidth{\@tocpagenum{#7}}\par
    \nobreak
    \endgroup
  \fi}
\makeatother
\AtBeginDocument{%
\makeatletter
\expandafter\renewcommand\csname r@tocindent0\endcsname{0pt}
\makeatother
}
\def\l@subsection{\@tocline{2}{0pt}{2.5pc}{5pc}{}}

\renewcommand{\theequation}{\arabic{section}.\arabic{equation}}

\newtheorem{theorem}{Theorem}[section]
\newtheorem{lemma}[theorem]{Lemma}
\newtheorem{corollary}[theorem]{Corollary}
\newtheorem{proposition}[theorem]{Proposition}

\numberwithin{equation}{section}

\theoremstyle{remark}
\newtheorem{remark}{Remark}[section]

\allowdisplaybreaks[3]
\newcommand{\Z}{{\mathbb{Z}}}

\newcommand{\R}{{\mathbb{R}}}

\newcommand{\T}{{\mathbb{T}}}
\newcommand{\C}{\mathbb{C}}

\newcommand{\n}[1]{\|#1\|}
\newcommand{\bn}[1]{\big\|#1\big\|}
\newcommand{\bbn}[1]{\Big\|#1\Big\|}
\newcommand{\lr}[1]{\left\{#1\right\}}
\newcommand{\lrs}[1]{\left(#1\right)}

\def\leq{\leqslant}
\def\geq{\geqslant}

\def\la{\lambda}
\def\om{\omega}

\def\na{\nabla}
\def\De{\Delta}
\def\pa{\partial}
\def\pat{\partial_t}

\def\lan{\langle}
\def\ran{\rangle}

\def\im{\operatorname{Im}}

\begin{document}
\title[Quantitative stability for the 2D Couette flow]{Quantitative stability for the 2D Couette flow on the infinite channel with non-slip boundary condition}

\author[]{Qionglei Chen}
\address[]{Institute of Applied Physics and Computational Mathematics, 100088 Beijing, China}
\email{chen\_qionglei@iapcm.ac.cn}

\author[]{Zhen Li}
\address[]{School of Mathematical Sciences, Key Laboratory of Mathematics and Complex Systems, Ministry of Education, Beijing Normal University, 100875 Beijing, China}
\email{lizhen@bnu.edu.cn}

\author[]{Changxing Miao}
\address[]{Institute of Applied Physics and Computational Mathematics, 100088 Beijing, China}
\email{miao\_changxing@iapcm.ac.cn}
\begin{abstract}
In this paper, we investigate the quantitative stability for the 2D Couette flow on the infinite channel $\R\times [-1,1]$ with non-slip boundary condition. Compared to the case $\T\times [-1,1]$, we  establish the stability in the context of long wave associated with the frequency range $0\leq |k|<1$ by developing the resolvent estimate argument.
The new ingredient is  to discover the key division point at $10\nu$ in the frequency interval $(0,1)$ by  the sharp Sobolev constant in Wirtinger's inequality together with the refined estimates of the Airy function in the interval $(0,1)$, and then we  establish the space-time estimates on the low-frequency $0\leq |k|\leq 10 \nu$ and the intermediate-frequency $ 10 \nu\leq |k|<1$, respectively.  As an application of the space-time estimates, we  obtain the nonlinear transition threshold to be $\gamma\leq\frac12$.
Meanwhile, we also show that when the frequencies $|k|\geq \nu^{1-}$, the enhanced dissipation effect occurs for the linearized Navier-Stokes equations.
\end{abstract}

\maketitle
\tableofcontents
\section{Introduction}\label{sec:Introduction}
We consider the incompressible Navier-Stokes equations on the infinite channel $\Omega=\R\times \mathrm{I}$ with $\mathrm{I}=[-1,1]$:
\begin{equation}\label{NS}
	\left\{
    \begin{aligned}
    &\pat V + (V\cdot\na)V -\nu\De V+ \na P =0, \\
    &\na\cdot V = 0,\\
    & V(0,x,y) = V^{in}(x,y),
    \end{aligned}
	\right.
	\end{equation}
where $V(t,x,y)$ is the velocity field, $P$ is the corresponding pressure and $\nu>0$ is the viscosity. Understanding the enhanced dissipation and transition threshold in shear flows may be a crucial step in addressing the transition from laminar to turbulent flow, see \cite{Reynolds,Rayleigh,Kelvin}.

The Couette flow, among the simplest laminar flows, is a steady solution to the Navier-Stokes equation \eqref{NS}, yet it induces several long-standing problems in hydrodynamics. Notably, the Sommerfeld paradox indicates that theoretical analysis shows that the plane Couette flow and pipe Poiseuille flow are linearly stable for all Reynolds number $Re>0$ \cite{Drazin-Reid,Romanov}; however, the experiments and numerical simulations \cite{Chapman,Yaglom} have revealed that these flows transition to turbulence beyond a critical Reynolds number. An increasing amount of effort has been dedicated to understanding this paradox, including Chapman \cite{Chapman} and references therein.  Trefethen et al. \cite{Trefethen} initially reformulated the problem into the study of the transition threshold. Later, Bedrossian-Germain-Masmoudi \cite{Bedrossian-Germain-Masmoudi-4} presented a mathematical formulation of the transition threshold problem as described below:

Find a suitable space $X$ with the norm $\|\cdot\|_{X}$ and determine a $\gamma=\gamma(X)$ so that
\begin{align}
&\|u^{in}\|_{X}\leq \nu^{\gamma}\Longrightarrow \text{ stability},\label{stable}\\
&\|u^{in}\|_{X}\gg \nu^{\gamma}\Longrightarrow \text{ instability}.\label{unstable}
\end{align}
The exponent $\gamma$ is referred to as the transition threshold.

A great deal of work has been devoted to studying the transition threshold problem for the Couette flow. Romanov \cite{Romanov} first obtained the linear spectrum stability for Couette flow in three-dimensional domain $\{x,y,z: |x|\leq 1, y\in \R, z\in\R\}$ with non-slip boundary condition. In a series of important works \cite{Bedrossian-Germain-Masmoudi-1,Bedrossian-Germain-Masmoudi-2,Bedrossian-Germain-Masmoudi-3,
Bedrossian-Germain-Masmoudi-4,Bedrossian-Masmoudi-Vicol,Bedrossian-Wang-Vicol,Masmoudi-Zhao}, Bedrossian, Germain, Masmoudi et al. studied the stability issue of the Couette flow and obtained significant achievements. Chen-Li-Wei-Zhang \cite{Chen-Li-Wei-Zhang} first investigated the Couette flow on $\T\times [-1,1]$ with the non-slip boundary condition and acquired the transition threshold $\gamma\leq \frac12$ by establishing the resolvent estimates and the space-time estimates. Arbon-Bedrossian \cite{Arbon-Bedrossian} first derived the transition threshold $\gamma\leq \frac12+$ on three unbounded $x$-domains: the whole plane $\R\times \R$, the half plane $\R\times [0,\infty)$ and $\R\times [-1,1]$ with Navier-slip boundary conditions. For more study of Couette flow in different settings, see \cite{Chen-Wei-Zhang2,Wei-Zhang-1,Wei-Zhang-2,Li-Masmoudi-Zhao} and references therein.

In addition to the Couette flow, the Poiseuille flow and Kolmogorov flow have also attracted the attention of many researchers, see \cite{Ibrahim,Li-Wei-Zhang1,Wei-Zhang-Zhao-2,Chen-Wei-Zhang,Chen-Ding-Lin-Zhang, Zelati-Elgindi-Widmayer,Zotto,Ding-Lin,Wei-Zhang}. Furthermore, for other mathematical results concerning the stability of general shear flows, one may refer to \cite{Almog-Helffer,Chen-Wei-Zhang1,Gerard-Maekawa-Masmoudi,Lin-Xu, Wei-Zhang-Zhao,Wei-Zhang-Zhao-1,Zillinger,Zillinger1} and references therein.

Although there is a mounts of studies on the stability in the case $x\in \T$,  there has been little related mathematical research on $x \in \R$. It is an interesting issue to study the long waves effect on nonlinear stability, which may lead to instability and potentially give rise to turbulence.  Nevertheless, there seems to be no quantitative analysis for Couette flow on $\R\times [-1,1]$ with non-slip boundary condition. This motivates us to study the transition threshold and enhanced dissipation of the Couette flow
$$U_{Cou}=(y,0)$$
related to \eqref{NS} on the infinite channel $\Omega=\R\times \mathrm{I}$.
For this purpose, we consider the perturbation velocity $u=V-U$ such that
\begin{equation}\label{pertu}
\left\{
\begin{aligned}
&\pat u+u\cdot\na u-\nu \De u+(u^2,0)+y\pa_xu+\na p=0,\\
&\na\cdot u=0,\\
&u(t,x,\pm 1)=0,\\
&u(0,x,y)=u^{in}(x,y),
\end{aligned}
\right.
\end{equation}
and prove the transition threshold $\gamma \leq \frac12$ so that \eqref{stable} is valid for an appropriate space, where the non-slip boundary condition given in \eqref{pertu} exhibits the boundary layer effect.

Let $\om=\pa_y u^1-\pa_x u^2$ be the vorticity and $\psi$ be the stream function. Then the vorticity system can be written as
\begin{equation}\label{equ: omli11}
\left\{
\begin{aligned}
&\pat\om-\nu\De \om+u\cdot\na \om+y\pa_x\om=0,\\
&\De\psi=\om,\\
&\psi(t,x,\pm 1)=\pa_y \psi(t,x,\pm 1)=0,\\
&\om(0,x,y)=\om^{in}(x,y).
\end{aligned}
\right.
\end{equation}

Given a function $f(t,x,y)$, we denote its Fourier transform with respect to the $x$-variable by
\begin{align*}
f_k(t,y)=:\int_{\R} f(t,x,y) e^{-ikx}dx.
\end{align*}
We take the Fourier transform in $x$-variable for the linearized system of \eqref{equ: omli11}:
\begin{equation}
\left\{
\begin{aligned}\label{equ: omli11k}
&\pa_t\om_k-\nu(\pa^2_y-k^2)\om_k+ik y\om_k=0, \\
&\om_k=(\pa^2_y-k^2)\psi_k,\\
&\psi_k(t,k,\pm 1)=\psi'_k(t,k,\pm 1)=0,\\
&\om_k(0)=\om^{in}_k(k,y),
\end{aligned}
\right.
\end{equation}
where $k\neq 0$.

Now, we are in a position to state the main result.
\begin{theorem}\label{Th: tran thre}
Suppose that $u^{in}\in H^1_0\cap H^2(\Omega)$ with $\mathrm{div} u^{in}=0$ in \eqref{pertu}. Let $0<\nu_0\leq 1$ and $c$ be a suitably small constant. Then for $0<\nu\leq\nu_0$ and $\|u^{in}\|_{H^2}\leq c\nu^{\frac12}$, the solution $u$ to \eqref{pertu} satisfies
\begin{align*}
\| E_k\|_{L^1_k(\R)} \leq C\nu^{\frac12},
\end{align*}
where
\begin{equation*}
E_k=\left\{
\begin{aligned}
&\|\om_k\|_{L^\infty_t L^2_y}+\nu^{\frac12}\|\om_k\|_{L^2_t L^2_y}+\|u_k\|_{L^\infty_t L^\infty_y}+\nu^{\frac12}\|u_k\|_{L^2_tL^2_y}, \quad 0\leq|k|< 10\nu,\\
&\|(1-|y|)^{\frac12}\om\|_{L^\infty_t L^2_y}+\nu^{\frac14}|k|^{-\frac14}\|\om_k\|_{L^\infty_t L^2_y}+\nu^{\frac14}|k|^{\frac14}\|\om_k\|_{L^2_t L^2_y}+\|u_k\|_{L^\infty_t L^\infty_y}\\
&+|k|^{\frac12}\|u_k\|_{L^2_tL^2_y},\qquad\qquad\qquad\qquad\qquad\qquad\qquad\qquad\qquad \qquad  10\nu\leq|k|<1,\\
&\|(1-|y|)^{\frac12}\om_k\|_{L^\infty_t L^2_y}+\nu^{\frac14}|k|^{\frac12}\|\om_k\|_{L^2_t L^2_y}+|k|^{\frac12}\|u_k\|_{L^\infty_t L^\infty_y}+|k|\|u_k\|_{L^2_t L^2_y}, \,\, |k|\geq 1.
\end{aligned}
\right.
\end{equation*}
\end{theorem}
\begin{corollary}
Under the same conditions as Theorem \ref{Th: tran thre}, if we additionally have $u^{in}\in L^1_x (H^1_{0,y}\cap H^2_y)$, then there holds
\begin{align*}
\| E_k\|_{L^\infty_k(\R)} \leq C\nu^{\frac12}.
\end{align*}
\end{corollary}

As an application of space-time estimates in Theorems \ref{Th xi small} and \ref{Th: linear problem}, we discover that when the frequency range $|k|\geq \nu^{1-}$, the solution of \eqref{equ: omli11k} exhibits the enhanced dissipation effect.
\begin{theorem}\label{Th: linear problem enhance}
Let $\om_k$ be a solution of \eqref{equ: omli11k} with $\om^{in}_k\in H^1_y$. Then there exist constants $\epsilon_0>0$ suitably small and $0<\nu_0\leq 1$ so that for $0<\nu\leq \nu_0$, we have
\begin{align*}
\max\{1,|k|\}\|e^{\epsilon_0\la_{\nu} t}u_k\|_{L^\infty_t L^2_y}+\max\{\nu^{\frac12},|k|^{\frac12},|k|\}\|e^{\epsilon_0\la_{\nu} t}u_k\|_{L^2_t L^2_y} \leq \|\om^{in}_k\|_{H^1_y},
\end{align*}
where $u_k=(\pa_y (\pa^2_y-k^2)^{-1}\om_k,-ik(\pa^2_y-k^2)^{-1}\om_k)$ and $\la_{\nu}=\max\{\nu, (\nu |k|^2)^{\frac13}\}$.
\end{theorem}
\begin{remark}
Actually, we can obtain the same enhanced dissipation estimates as shown in Theorem \ref{Th: linear problem enhance} for the Navier-slip boundary condition.
\end{remark}

 Now, we summarize the key contributions of our results:
\begin{itemize}
  \item The transition threshold stated in Theorem \ref{Th: tran thre} matches the known optimal result on $\T\times [-1,1]$ in \cite{Chen-Li-Wei-Zhang}, in which the authors only require to deal with the space-time estimates on high-frequency $|k|\geq 1$.
      Compared to $\T\times [-1,1]$,  the domain $\R\times [-1,1]$ exists extra difficulty on frequency range $0\leq |k|<1$. More precisely, we are unable to establish a uniform estimate on the this range. So we will adopt a ``divide and conquer" strategy to address the range $0\leq |k|\leq 1$ by the following observations:
\begin{itemize}
\item For the low-frequency $ 0\leq|k|< 10\nu$,  utilizing the sharp Sobolev constant from Wirtinger's inequality and the expansion argument in $L^2(-1,1)$, we identify the key point $\frac{\pi^3}{3}\nu\sim 10\nu$, which allows us to construct the space-time estimates in $0\leq|k|<10\nu$  by the energy method. To some extent, this case can be viewed as a perturbation result around $k=0$ and hence the range cannot be extended to   $0\leq|k|\leq 1$.
 \item For the intermediate-frequency $10\nu \leq |k|<1$, we explore a kind of asymptotic behavior of $A_0(z)$ produced by the Airy function to achieve the space-time estimates when $|k|\leq 1$. Compared to $|k|\geq1$ in \cite{Chen-Li-Wei-Zhang}, the new insight is that the maximized gain of $\sigma$ in Lemma \ref{A3} enables us to extend the lower bound of intermediate-frequencies to $10\nu$.
 \item We notice that the constant $10\nu$ is of inconsequence for the Navier-slip boundary condition. More precisely, we can use energy method and resolvent estimate to obtain the space-time estimate on the frequency range $0\leq |k|\leq 10\nu$  and $\nu\leq |k|<1$, respectively.
\end{itemize}
\item It should be pointed out that the intermediate frequency will disappear on $\T\times [-1,1]$ due to the constraint $k\in\Z$. For the domain $\R\times [-1,1]$, there are nine kinds of interactions between the low, intermediate and high frequencies of $u$ and $\om$. By employing a class of bilinear estimates and optimization argument, we detect the expected form on energy functional, and ultimately achieve transition threshold $\gamma \leq \frac12$.
\end{itemize}

\vspace{0.2cm}
\noindent \textbf{Outline:}
 To investigate the  nonlinear stability and enhanced dissipation, a key step is to derive the space-time estimates for the linearized Navier-Stokes equation
\begin{align}\label{equ: omkli}
\pa_t\om-\nu(\pa^2_y-k^2)\om+i k y\om=0, \qquad \om(0)=\om^{in}( k,y),
\end{align}
where $\om=(\pa^2_y - k^2)\psi$ with $\psi(\pm 1)=\psi'(\pm 1)=0$, $k\in \R$. In order to understand the linear stability, the standard method is to seek the solution of \eqref{equ: omkli} in the form
\begin{equation*}
w_k=\om_k(y)e^{-ikt\la},
\end{equation*}
where $w_k$ satisfies the following Orr-Sommerfeld equation
\begin{align}\label{OS eq}
-\nu(\pa_y^2-k^2) w_k+ik (y-\la) w_k=0.
\end{align}
If $\la\in \C$ , all the nontrivial solutions of \eqref{OS eq} with the non-slip or Navier-slip boundary conditions lie in $\im k\la<0$. Then, the Couette flow is linearly spectral stable.

Firstly, we investigate the resolvent estimates for the linearized operator
$$\mathcal{L}=-\nu(\pa^2_y-k^2)+ik y,$$
which is reduced to the study of the inhomogeneous Orr-Sommerfeld (OS) equation
\begin{equation}\label{equ: psij1}
\left\{
\begin{aligned}
&-\nu(\pa^2_y-k^2) w+i k(y-\la)w=F,\\
&w=(\pa^2_y-k^2)\phi,\\
&\phi(\pm 1)=0,\qquad \phi'(\pm 1)=0.
\end{aligned}
\right.
\end{equation}
We decompose the solution $w$ into three parts:
$$w=w_{Na}+c_1w_{1}+c_2w_{2},$$
where $w_{Na}$ is the solution of the inhomogeneous OS equation with the Navier-slip boundary condition and $w_1$, $w_2$ are the solutions of the homogeneous OS equation. The estimate of $w_{Na}$ are derived based on the resolvent estimates by multiplies methods in \cite{Li-Wei-Zhang,Li-Wei-Zhang1}. The estimates for $w_1$ and $w_2$ rely on the refined estimates of the Airy function provided in Section \ref{sec: reso esti}.

Secondly, we are devoted to establishing the space-time estimates.
We first divide the frequency $k\in \R$ into three parts:
\begin{itemize}
\item Low-frequency: $\{k: 0\leq |k|<10\nu\}$,
\item Intermediate-frequency: $\{k: 10\nu\leq |k|<1\}$,
\item High-frequency: $\{k:  |k|\geq 1\}$.
\end{itemize}

For the low-frequency, decomposing the stream function through the $L^2$ basis and using the sharp Sobolev constant from Wirtinger's inequality,  we obtain the space-time estimates for the maximum frequency range $0\leq |k|<\frac{\pi^3}{3}\nu$ by the energy method.

For the intermediate-frequency, we reformulate the solution $\om$ in terms of homogeneous and inhomogeneous equations, and then derive the space-time estimates based on the resolvent estimates in Section \ref{sec: reso esti}. To extend the lower bound of the intermediate-frequency to $10\nu$, it requires a series of estimates on the Airy function.

Thirdly, by making use of the space-time estimates  together with a delicate analysis on the interactions between the low, intermediate and high frequencies of $u$ and $\om$, we establish the quantitative stability for the 2D Couette flow on the infinite channel $\R\times [-1,1]$ with non-slip boundary condition. As an application of the space-time estimates, we establish the enhanced dissipation estimates on the frequency range $|k|\geq \nu^{1-}$.

\vspace{0.2cm}

\noindent{\bf Notations.} Throughout this paper, $C$ denotes a general constant independent of $\nu, k, \la$, and it may vary from line to line. The notation $A\lesssim B$ means $A\leq CB$.
In the following, we will omit the subscript $k$ for $\om, w, \psi, \phi, F$ while still keep the dependence on $k$ in the actual estimates.

\section{The Resolvent Estimates of the Linear Operator}\label{sec: reso esti}
In this section, we study the linear operator
$$\mathcal{L}=-\nu(\pa^2_y-k^2)+ik y$$
and establish the resolvent estimates of $\mathcal{L}$ with the Navier-slip boundary condition in Subsection \ref{sec: navier-slip} and non-slip boundary condition in Subsection \ref{sec: Non slip}, respectively.
\subsection{The resolvent estimates with the Navier-slip boundary condition}\label{sec: navier-slip}
Consider the equation
\begin{equation}\label{equ: psiNa}
\left\{
\begin{aligned}
&-\nu(\pa^2_y-k^2)w +ik(y-\la)w -\epsilon\nu^{\frac13}|k|^{\frac23}w =F ,\\
&w (\pm 1)=0,
\end{aligned}
\right.
\end{equation}
where $\la\in \mathbb{R}$ and $\epsilon\geq 0$ is sufficiently small and independent of $\nu, k, \la$. We show the following resolvent estimates under the Navier-slip boundary condition with intermediate-frequency $10\nu\leq |k|<1$.
\begin{proposition}\label{pro: reso esti FL2}
Let $w$ be a solution of \eqref{equ: psiNa} with $\la\in \R$, $10\nu\leq |k|<1$.  Denote $u=(\pa_y(\pa^2_y-k^2)^{-1}w,-ik(\pa^2_y-k^2)^{-1}w)$. If $F\in L^2(I)$, there exists $\epsilon_0<1$ such that for $0\leq \epsilon \leq \epsilon_0$, we have
\begin{equation}
\begin{aligned}\label{lemma: wL2FL2}
\nu^{\frac16}|k|^{\frac56}\|u\|_{L^2}+\nu^{\frac16}|k|^{\frac56}\|w\|_{L^1}+\nu^{\frac23}|k|^{\frac13}\|w'\|_{L^2}&\\
+(\nu|k|^2)^{\frac13}\|w\|_{L^2}+|k|\|(y-\la)w\|_{L^2}&\leq C\|F\|_{L^2}.
\end{aligned}
\end{equation}
If $F\in H^{-1}(I)$, there exists $\epsilon_1<1$ such that for $0\leq \epsilon \leq \epsilon_1$, we have
\begin{align}\label{lemma: wL2,FH-1}
\nu^{\frac12} |k|^{\frac12}\|u \|_{L^2}+\nu \|w'\|_{L^2}+\nu^{\frac23}|k|^{\frac13}\|w \|_{L^2}\leq C\|F \|_{H^{-1}}.
\end{align}
\end{proposition}

\textbf{Sketch of the proof:} The main difficulty in deriving the resolvent estimate stems from the fact that the operator $\mathcal{L}$ may not be invertible at $y=0$. By a scaling analysis with $y-\la\sim \delta$, we have
\begin{align*}
\nu \delta^{-2} w\sim k\delta w\sim k(y-\la)w\sim F,
\end{align*}
which yields $\delta\sim(\nu k^{-1})^{\frac13}$. Noting that for $10\nu\leq |k|<1$, the fact $\delta<1$ still holds true. By a similar proof of Proposition 3.1--Lemma 3.5 in \cite{Chen-Li-Wei-Zhang}, we can obtain
\begin{align*}
&\nu^{\frac16}|k|^{\frac56}\|w\|_{L^1}+\nu^{\frac23}|k|^{\frac13}\|w'\|_{L^2}
+(\nu|k|^2)^{\frac13}\|w\|_{L^2}+|k|\|(y-\la)w\|_{L^2}\leq C\|F\|_{L^2},\\
&\nu \|w'\|_{L^2}+\nu^{\frac23}|k|^{\frac13}\|w \|_{L^2}\leq C\|F \|_{H^{-1}},
\end{align*}

The first term in \eqref{lemma: wL2FL2} and \eqref{lemma: wL2,FH-1} does not share the same scale as the others. Denote
\begin{align*}
\phi=(\pa^2_y-k^2)^{-1}w.
\end{align*}
Thanks to $\|u\|_{L^2}=\|\pa_y\phi\|_{L^2}+\|k\phi\|_{L^2}$, it is easy to check that $\|\pa_y \phi\|_{L^2}$ is the main part of $\|u\|_{L^2}$ when $10\nu\leq |k|<1$. Consequently, in contrast to \cite{Chen-Li-Wei-Zhang}, we can eliminate $|k|^{\frac12}$-factor for the terms including $\|u\|_{L^2}$ by employing the inequality $ \|\pa_y\phi\|_{L^2}\|\phi\|_{L^2}\leq \|u\|^2_{L^2}$, rather than $ \|\pa_y\phi\|_{L^2}\|k\phi\|_{L^2}\leq \|u\|^2_{L^2}$ used in \cite{Chen-Li-Wei-Zhang}.

\vspace{0.2cm}
In a similar way, we can extend the weak resolvent estimate in \cite{Chen-Li-Wei-Zhang} to the frequency range $10\nu\leq |k|<1$, which plays an important role in deriving the resolvent estimate under the non-slip boundary condition in Proposition \ref{lemma:non-slip boundary,resolvent}.

\begin{proposition}\label{lemma: weak resolvent estimate}
Let $w$ be as in Proposition \ref{pro: reso esti FL2}. For $10\nu\leq |k|<1$, $f\in H^1(I)$ and $f(-j)=0$ with $j\in\{\pm 1\}$, there holds
\begin{align*}
|\lan w,f\ran|\leq &C| k|^{-1}\|F\|_{H^{-1}}\Big(\delta^{-\frac32}\|f\|_{L^\infty((-1,1)\cap(\la-\delta,\la+\delta))}\\
&+|f(j)|(|j-\la|+\delta)^{-\frac34}\delta^{-\frac34} +\|f\chi\|_{H^1}+\delta^{-1}\|f\chi\|_{L^2} \Big),
\end{align*}
where $\delta=\nu^{\frac13}|k|^{-\frac13}$ and $\chi$ is a cut-off function defined by
\begin{equation}\label{fun: chi}
\chi(y)=\left\{
\begin{aligned}
&\frac{1}{y-\la}, \qquad  y\in(-1,1)\cap [(-1,\la-\delta)\cup (\la+\delta,1)],\\
&2\frac{y-\la}{\delta^2}-\frac{(y-\la)^3}{\delta^4} \qquad y\in (-1,1)\cap(\la-\delta,\la+\delta).
\end{aligned}
\right.
\end{equation}
\end{proposition}
\subsection{Resolvent estimates with non-slip boundary condition}\label{sec: Non slip}
 Denote
\begin{align*}
w=(\pa^2_y-k^2)\phi, \quad u=(\pa_y\phi, -ik\phi).
\end{align*}
The stream function satisfies
\begin{equation}\label{equ: psij}
\left\{
\begin{aligned}
&-\nu(\pa^2_y-k^2)^2\phi+i k(y-\la)(\pa^2_y-k^2)\phi-\epsilon\nu^{\frac13}| k|^{\frac23}(\pa^2_y-k^2)\phi=F,\\
&\phi(\pm 1)=0,\qquad \phi'(\pm 1)=0,
\end{aligned}
\right.
\end{equation}
where $\la\in \mathbb{R}$ and $\epsilon\geq 0$ is sufficiently small and independent of $\nu, k, \la$.
Define
\begin{equation}\label{def:rhok}
\rho_k(y)=\left\{
\begin{aligned}
&\quad 1, \qquad\quad  |y|\leq 1-\delta, \\
&\frac{1-|y|}{\delta},\quad 1-\delta\leq |y|\leq1,
\end{aligned}
\right. \qquad \delta=\nu^{\frac13}|k|^{-\frac13}.
\end{equation}
Now, we give the resolvent estimate with the non-slip boundary condition.
\begin{proposition}\label{lemma:non-slip boundary,resolvent}
Let $\phi$ be a solution of \eqref{equ: psij} with $10\nu\leq |k|<1$. If $F\in L^2(I)$, then we have
\begin{align}\label{wL1+wL2}
\nu^{\frac13}|k|^{\frac23}\|\rho^{\frac12}_k w\|_{L^2}+\nu^{\frac16}|k|^{\frac56}\|u\|_{L^2}+\nu^{\frac{5}{12}}|k|^{\frac{7}{12}}\|w\|_{L^2}\leq C\|F\|_{L^2}.
\end{align}
If $F\in H^{-1}$, then we have
\begin{align}\label{wL2+rhoL2+uL2}
\nu^{\frac23}|k|^{\frac13}\|\rho^{\frac12}_k w\|_{L^2}+\nu^{\frac12}|k|^{\frac12}\|u\|_{L^2}+\nu^{\frac34}|k|^{\frac14}\|w\|_{L^2}\leq C\|F\|_{H^{-1}}.
\end{align}
\end{proposition}

\noindent\textbf{Reformulation of the problem.}
We decompose the solution $\phi$ of \eqref{equ: psij} as follows:
\begin{align*}
\phi=\phi_{Na}+c_{1}\phi_{1}+c_{2}\phi_{2},
\end{align*}
where $\phi_{Na}$ solves the inhomogeneous OS equation with the Navier-Slip boundary condition
\begin{equation}\label{equ: psiNa2}
\left\{
\begin{aligned}
&-\nu(\pa^2_y-k^2)^2\phi_{Na}+i k(y-\la)(\pa^2_y-k^2)\phi_{Na}-\epsilon\nu^{\frac13}| k|^{\frac23}(\pa^2_y-k^2)\phi_{Na}=F,\\
&\phi_{Na}(\pm 1)=0,\qquad \phi''_{Na}(\pm 1)=0,
\end{aligned}
\right.
\end{equation}
and $\phi_1$, $\phi_2$ solve the following homogeneous OS equations
\begin{equation}\label{equ: psi1}
\left\{
\begin{aligned}
&-\nu(\pa^2_y-k^2)^2\phi_{1}+i k(y-\la)(\pa^2_y-k^2)\phi_{1}-\epsilon\nu^{\frac13}| k|^{\frac23}(\pa^2_y-k^2)\phi_{1}=0,\\
&\phi_{1}(\pm 1)=0,\qquad \phi'_{1}( 1)=1,\quad \phi'_{1}(-1)=0,
\end{aligned}
\right.
\end{equation}
and
\begin{equation}\label{equ: psi2}
\left\{
\begin{aligned}
&-\nu(\pa^2_y-k^2)^2\phi_{2}+i k(y-\la)(\pa^2_y-k^2)\phi_{2}-\epsilon\nu^{\frac13}| k|^{\frac23}(\pa^2_y-k^2)\phi_{2}=0,\\
&\phi_{2}(\pm 1)=0,\qquad \phi'_{2}(-1)=1,\quad \phi'_{2}(1)=0.
\end{aligned}
\right.
\end{equation}
Let $w_{Na}=(\pa_y^2- k^2)\phi_{Na}$ and $w_{i}=(\pa^2_y- k^2)\phi_i$ with $i\in\{1,2\}$. Then we have
\begin{align*}
w=w_{Na}+c_{1} w_{1}+c_{2}w_{2}=(\pa^2_y- k^2)\phi.
\end{align*}
Utilizing the boundary conditions in \eqref{equ: psi1} and \eqref{equ: psi2}, we obtain
\begin{equation}\label{def: c1c2}
\begin{aligned}
c_{1}(\la)=&\int^1_{-1}\frac{\sinh  k(y+1)}{\sinh 2 k}w_{Na}(y)dy,\\
c_{2}(\la)=&\int^1_{-1}\frac{\sinh  k(1-y)}{\sinh 2 k}w_{Na}(y)dy.
\end{aligned}
\end{equation}
The detailed calculations of $c_1$ and $c_2$ can be seen in \cite[P142-143]{Chen-Li-Wei-Zhang}.

The proof of Proposition \ref{lemma:non-slip boundary,resolvent} relies on not only the resolvent estimate under the Navier-slip boundary condition in Proposition \ref{pro: reso esti FL2}, but also the subsequent estimates on $w_{j}$ and the coefficients $c_{j}$ with $j=1,2$.

\vspace{0.2cm}
\noindent\textbf{Estimates of $w_1$ and $w_2$.}
It is well known that
\begin{align*}
W_1(y)=&Ai(e^{i\frac{\pi}{6}}(L(y-\la-ik\nu)+i\epsilon)),\\
W_2(y)=&Ai(e^{i\frac{5\pi}{6}}(L(y-\la-ik\nu)+i\epsilon))
\end{align*}
are two linearly independent solutions of the homogeneous OS equation
\begin{align*}
-\nu(w''-k^2w)+ik(y-\la)w-\epsilon\nu^{\frac13}|k|^{\frac23}w=0,
\end{align*}
where $Ai(y)$ is the Airy function given by a nontrivial solution of $f''-yf=0$.
In terms of the boundary conditions, the solutions of \eqref{equ: psi1} and \eqref{equ: psi2} can be rewritten as
\begin{align}\label{expre: w1,w2}
w_1=C_{11}W_1+C_{12}W_2, \quad w_2=C_{21}W_1+C_{22}W_2,
\end{align}
where the coefficients $C_{ij}$ $(1\leq i,j\leq 2)$ satisfy
\begin{equation*}
\left(\begin{aligned}
&C_{11}\\
&C_{12}
\end{aligned}
\right)=\frac{\left(
\begin{aligned}
&A_2e^{ k}-B_2e^{- k}\\
&-B_1e^{ k}+A_1e^{- k}
\end{aligned}
\right)}{A_1A_2-B_1B_2},\qquad
\left(\begin{aligned}
&C_{21}\\
&C_{22}
\end{aligned}
\right)=\frac{\left(
\begin{aligned}
&-A_2e^{- k}+B_2e^{ k}\\
&B_1e^{- k}-A_1e^{ k}
\end{aligned}
\right)}{A_1A_2-B_1B_2}
\end{equation*}
with
\begin{equation}\label{A12B12}
\begin{aligned}
&A_1=\int^1_{-1}e^{ky}W_1(y)dy,\qquad A_2=\int^1_{-1} e^{-ky}W_2(y)dy,\\
&B_1=\int^1_{-1}e^{-ky}W_1(y)dy,\qquad B_2=\int^1_{-1}e^{ky}W_2(y)dy.
\end{aligned}
\end{equation}
Next, we give the estimates of $C_{ij}$ and $W_j$ with $1\leq i,j\leq 2$ and $10\nu\leq |k|<1$. Denote
\begin{align*}
A_0(z)=\int^{\infty}_{e^{i\pi/6}z} Ai(t) dt,\quad  d=-1-\la-i k \nu,\quad  \widetilde{d}=-1+\la-i k \nu,\quad L=({k}/{\nu})^{\frac13}.
\end{align*}
\begin{lemma}\label{lemma: W1W2}
For $10\nu\leq |k|<1$, there exists $\epsilon_0$ such that for $ \epsilon\leq \epsilon_0$, it holds
\begin{align}
&\frac{L}{|A_0(Ld+i\epsilon)|}\|W_1\|_{L^\infty}\leq CL(1+|L(1+\la)|^{\frac12}),\label{W1Linfy}\\
&\frac{L}{|A_0(L\widetilde{d}+i\epsilon)|}\|W_2\|_{L^\infty}\leq CL(1+|L(1-\la)|^{\frac12}),\label{W2Linfty}\\
&\frac{L}{|A_0(Ld+i\epsilon)|}\|W_1\|_{L^1}+\frac{L}{|A_0(L\widetilde{d}+i\epsilon)|}\|W_2\|_{L^1}\leq C,\label{W1+W2L1}\\
&\frac{L}{|A_0(Ld+i\epsilon)|}\|\rho^{\frac12}_k W_1\|_{L^2}+\frac{L}{|A_0(L\tilde{d}+i\epsilon)|}\|\rho^{\frac12}_k W_2\|_{L^2}\leq CL^{\frac12}.\label{we-W1+W2}
\end{align}
\end{lemma}
\begin{proof}
Noting that $L=(k/\nu)^{1/3}>1$ in the case of $10\nu \leq |k| < 1$, we can
obtain the estimates \eqref{W1Linfy}-\eqref{we-W1+W2} for the frequency range
 $ 10\nu\leq|k|\le1$ in exactly same way
as in \cite[Lemma 5.2]{Chen-Li-Wei-Zhang}.
\end{proof}

Now we consider the estimates of $C_{ij}$.  Although the proof strategy is close to the one in \cite[Lemma 5.2]{Chen-Li-Wei-Zhang},  we encounter  different scenarios and difficulties, so that we require to exploit more structure, and more employ  refined estimates of the Airy function for the frequencies $10\nu \leq |k| < 1$.
 \begin{lemma}\label{lemma: Cij}
For $10\nu\leq |k|<1$, there exist $\epsilon_0$ such that for $ \epsilon \leq \epsilon_0$, it holds that
\begin{align*}
&|C_{11}|\leq \frac{C L }{|A_0(Ld+i\epsilon)|},\quad |C_{12}|\leq \frac{CL}{|A_0(L\widetilde{d}+i\epsilon)|},\\
&|C_{21}|\leq \frac{CL}{|A_0(Ld+i\epsilon)|},\quad |C_{22}|\leq \frac{CL }{|A_0(L\widetilde{d}+i\epsilon)|}.
\end{align*}
\end{lemma}
\begin{proof}
According to the same calculation as in \cite[(8.3),(8.4),(8.6),(8.7)]{Chen-Li-Wei-Zhang}, we have
\begin{equation*}
\begin{aligned}
&A_1=A_0(Ld+i\epsilon)\frac{e^{- k-i\pi/6}}{L}\left[1-e^{2 k}\eta(Ld+i\epsilon,2L)+\frac{ k}{L}\int^{2L}_0 e^{\frac{ k t}{L}}\eta (Ld+i\epsilon,t)dt\right],\\
&\overline{A_2}=A_0(L\widetilde{d}+i\epsilon)\frac{e^{- k-i\pi/6}}{L}\left[1-e^{2 k}\eta(L\widetilde{d}+i\epsilon,2L)+\frac{ k}{L}\int^{2L}_0 e^{\frac{ k t}{L}}\eta (L\widetilde{d}+i\epsilon,t)dt\right],\\
&B_1=A_0(Ld+i\epsilon)\frac{e^{ k-i\pi/6}}{L}\left[1-e^{-2 k}\eta(Ld+i\epsilon,2L)-\frac{ k}{L}\int^{2L}_0 e^{-\frac{ k t}{L}}\eta (Ld+i\epsilon,t)dt\right],\\
& \overline{B_2}=A_0(L\widetilde{d}+i\epsilon)\frac{e^{ k-i\pi/6}}{L}\left[1-e^{-2 k}\eta(L\widetilde{d}+i\epsilon,2L)-\frac{ k}{L}\int^{2L}_0 e^{-\frac{ k t}{L}}\eta (L\widetilde{d}+i\epsilon,t)dt\right].
\end{aligned}
\end{equation*}
For $10\nu \leq |k|< 1$, by \eqref{A12B12} and Lemma \ref{lemma: W1W2}, we arrive at
\begin{align*}
|A_2e^{k}-B_2e^{-k}|=&\Big|\int^1_{-1}(e^{k(1-y)}-e^{-k(1-y)})W_2(y)dy\Big|\\
\leq& \int^1_{-1} C|k|(1-y)|W_2(y)|dy\leq C|k|L^{-1}|A_0(L\tilde{d}+i\epsilon)|.
\end{align*}
We claim
\begin{align}\label{claim App D}
|A_1A_2-B_1B_2|\geq  |k| L^{-2} |A_0(Ld+i\epsilon)| |A_0(L\widetilde{d}+i\epsilon)|.
\end{align}
This yields
\begin{align*}
|C_{11}|\leq \frac{CL}{ |A_0(Ld+i\epsilon)|}.
\end{align*}

Next, we prove the claim \eqref{claim App D}. It suffice to deal with the case $10\nu\leq k<1$, as the proof for the case $-1<k\leq -10\nu$ follows a similar argument.
Due to
\begin{align}\label{A1A2-B1B2}
|A_1A_2-B_1B_2|=\left|\frac{A_1A_2}{B_1B_2}-1\right||B_1B_2|,
\end{align}
it remains to give the low bound of $|\frac{A_1A_2}{B_1B_2}-1|$ and $|B_1B_2|$.
From Lemma \ref{A3}, we deduce that
\begin{equation}\label{B1B2 geq}
\begin{aligned}
|B_1|\geq& L^{-1}|A_0(Ld+i\epsilon)|e^{k}\Big(1-e^{-2 k}e^{-2a_{\sigma}L}-\frac{ k}{L}\int^{2L}_0 e^{-\frac{ k t}{L}-a_{\sigma}t} dt\Big)\gtrsim L^{-1}|A_0(Ld+i\epsilon)|,\\
|B_2|\geq&L^{-1}|A_0(L\widetilde{d}+i\epsilon)|e^{k}\Big(1-e^{-2 k}e^{-2a_{\sigma}L}-\frac{ k}{L}\int^{2L}_0 e^{-\frac{ k t}{L}-a_{\sigma}t} dt\Big)\gtrsim L^{-1}|A_0(L\widetilde{d}+i\epsilon)|.
\end{aligned}
\end{equation}

In order to estimate $|\frac{A_1A_2}{B_1B_2}-1|$, we divide $A_1$, $A_2$, $B_1$, $B_2$  as follows:
\begin{align*}
A_1=&A_0(Ld+i\epsilon)\frac{e^{- k-i\pi/6}}{L}(A_{11}+A_{12}),\quad A_2=\overline{A_0(L\widetilde{d}+i\epsilon)}\frac{e^{- k+i\pi/6}}{L}(A_{21}+A_{22}),\\
B_1=&A_0(Ld+i\epsilon)\frac{e^{ -k-i\pi/6}}{L}(B_{11}+B_{12}),\quad
B_2=\overline{A_0(L\widetilde{d}+i\epsilon)}\frac{e^{ -k+i\pi/6}}{L}(B_{21}+B_{22}),
\end{align*}
where
\begin{align*}
A_{11}&=:1-e^{2 k}\eta_1(2L),\quad A_{12}=:\frac{ k}{L}\int^{2L}_0 e^{\frac{ k t}{L}}\eta_1 (t)dt,\\
A_{21}&=:1-e^{2 k}\eta_2(2L),\quad A_{22}=:\frac{ k}{L}\int^{2L}_0 e^{\frac{ k t}{L}}\eta_2 (t)dt,\\
B_{11}&=:e^{2k}-\eta_1(2L),\quad B_{12}=:-\frac{ ke^{2k}}{L}\int^{2L}_0 e^{-\frac{ k t}{L}}\eta_1(t)dt,\\
B_{21}&=:e^{2k}-\eta_2(2L),\quad B_{22}=:-\frac{ ke^{2k}}{L}\int^{2L}_0 e^{-\frac{ k t}{L}}\eta_2(t)dt,
\end{align*}
with $\eta_1(t)=\eta(Ld+i\epsilon,t)$ and $\eta_2(t)=\overline{\eta (L\widetilde{d}+i\epsilon,t)}$.
It is easy to see that $A_{i2}$, $B_{i2}$ $(1\leq i\leq 2)$ are error terms. Denote
\begin{align*}
\left|\frac{A_1A_2}{B_1B_2}-1\right|=\left|\frac{(B_{11}+B_{12})(B_{21}+B_{22})-(A_{11}+A_{12})(A_{21}+A_{22})}{(B_{11}+B_{12})(B_{21}+B_{22})} \right|=:\left|\frac{D_1}{D_2}\right|.
\end{align*}

We first give the bound of $A_{ij}$ and $B_{ij}$ with $1\leq i,j\leq 2$. Thanks to Lemma \ref{A3}, we have
\begin{align*}
|A_{11}|+|A_{21}|\leq& 1+e^{2k-2a_{\sigma}L},\quad |A_{12}|+|A_{22}|\leq \frac{k}{L}\int^{2L}_{0}e^{\frac{ k t}{L}}e^{a_{\sigma} t}dt\leq \frac{k(1-e^{2k-2a_{\sigma}L})}{a_{\sigma}L-k},\\
|B_{11}|+|B_{21}|\leq& e^{2k}+e^{-2a_{\sigma}L}, \quad |B_{12}|+|B_{22}|\leq \frac{ ke^{2k}}{L}\int^{2L}_0 e^{-\frac{ k t}{L}}e^{a_{\sigma} t}dt \leq \frac{k(e^{2k}-e^{-2a_{\sigma}L})}{a_{\sigma}L+k}.
\end{align*}

For  $D_2$, in light of Lemma \ref{A3} and $10\nu\leq  k< 1$, we obtain
\begin{equation}\label{1-e2k}
\begin{aligned}
|D_2|=|(B_{11}+B_{12})(B_{21}+B_{22})|\leq (e^{2k}+e^{-2a_{\sigma}L})^2\left(\frac{k}{a_{\sigma}L+k}(e^{2k}-e^{-2a_{\sigma}L})\right)^2\leq e^8.
\end{aligned}
\end{equation}

We divide $D_1$ into
\begin{align*}
D_1=:\textrm{I}+\textrm{II}+\textrm{III},
\end{align*}
where
\begin{align*}
\textrm{I}=&B_{11}B_{21}-A_{11}A_{21},\\
\textrm{II}=&B_{11}B_{22}+B_{21}B_{12}-A_{11}A_{22}-A_{21}A_{12},\\
\textrm{III}=&B_{12}B_{22}-A_{12}A_{22}.
\end{align*}

 A direct computation shows
\begin{align*}
|\textrm{I}|=&|(e^{2k}-\eta_1(2L))(e^{2k}-\eta_2(2L))-(1-e^{2k}\eta_1(2L))(1-e^{2k}\eta_2(2L))|\\
=&|(e^{4k}-1)(1-\eta_1(2L)\eta_2(2L))|\\
\geq &4k(1-e^{-4a_{\sigma}L}),\\
|\textrm{II}|
\leq&\frac{2k(e^{2k}+e^{-2a_{\sigma}L})(e^{2k}-e^{-2a_{\sigma}L})}{a_{\sigma}L+k}+\frac{2k(1+e^{2k}e^{-2a_{\sigma}L})(1-e^{2k-2a_{\sigma}L})} {a_{\sigma}L-k} \\
= & \frac{2k(e^{4k}-e^{-4a_{\sigma}L})}{a_{\sigma}L+k}+\frac{2k(1-e^{4k-4a_{\sigma}L})}{a_{\sigma}L-k}\\
\leq& \frac{2k(e^{4k}-1+1-e^{-4a_{\sigma}L})}{a_{\sigma}L+k}+\frac{2k(1-e^{-4a_{\sigma}L})}{a_{\sigma}L-k}
\leq \frac{4k(1-e^{-4a_{\sigma}L})}{a_{\sigma}L-k}+\frac{8k^2}{a_{\sigma}L+k},
\end{align*}
and
\begin{align*}
|\textrm{III}|\leq & \frac{e^{4k}k^2(1-e^{-2a_{\sigma}L-2k})^2}{(a_{\sigma}L+k)^2}+\frac{k^2(1-e^{2k-2a_{\sigma}L})^2}{(a_{\sigma}L-k)^2}
\leq \frac{(e^{4k}+1)k^2}{(a_{\sigma}L-k)^2}.
\end{align*}
Inserting the estimates of $\textrm{I}$, $\textrm{II}$, $\textrm{III}$ into $D_1$, we obtain
\begin{align*}
|D_1|\geq 4k(1-e^{-4a_{\sigma}L})\left(1-\frac{1}{a_{\sigma}L-k}\right)-\frac{28k^2}{a_{\sigma}L+k}-\frac{(e^{4k}+1)k^2}{(a_{\sigma}L-k)^2}.
\end{align*}
It is clear that the maximum of $\frac{1}{a_{\sigma}L-k}$, $\frac{28k^2}{a_{\sigma}L+k}$ and $\frac{(e^{4k}+1)k^2}{(a_{\sigma}L-k)^2}$ are attained at $k=10\nu$. We choose $a_{\sigma}=0.47$ such that $1-\frac{1}{a_{\sigma}L-k}>0$ to ensure
\begin{align*}
|D_1|\geq 0.02k.
\end{align*}
This inequality together with \eqref{1-e2k} implies
\begin{align}\label{fracA1A2B1B2-1}
\left|\frac{A_1A_2}{B_1B_2}-1\right|=\left|\frac{D_1}{D_2}\right|\geq 0.002e^{-4}k.
\end{align}
Inserting \eqref{fracA1A2B1B2-1} and \eqref{B1B2 geq} into \eqref{A1A2-B1B2}, we arrive at
\begin{align*}
|A_1A_2-B_1B_2|\geq 0.002e^{-4}|k||B_1B_2|\gtrsim  |k| L^{-2} |A_0(Ld+i\epsilon)| |A_0(L\widetilde{d}+i\epsilon)|.
\end{align*}

In a similar way, we can obtain the bound of $C_{12}$, $C_{21}$ and $C_{22}$.
\end{proof}

Now, we estimate  $w_1$ and $w_2$, which is a direct result of Lemmas \ref{lemma: W1W2}--\ref{lemma: Cij} and interpolation.
\begin{proposition}\label{Lemma: Nonslip wLinfty wL1}
Let $w_1$ and $w_2$ be defined as in \eqref{expre: w1,w2}. For $10\nu\leq |k|<1$, there holds
\begin{align}
&\|w_1\|_{L^1}+\|w_2\|_{L^1}\leq C,\label{esti: w1+w2L1}\\
&\|w_1\|_{L^\infty}\leq C\nu^{-\frac12}| k|^{\frac12}(1+|1-\la|^{\frac12}),\label{w1Linfty}\\
&\|w_2\|_{L^\infty}\leq C\nu^{-\frac12}| k|^{\frac12}(1+|\la+1|^{\frac12}),\label{w2Linfty}\\
&\|w_1\|_{L^2}\leq C\nu^{-\frac14}| k|^{\frac14}(1+|1-\la|^{\frac12})^{\frac12},\label{esti: w1L2}\\
&\|w_2\|_{L^2}\leq C\nu^{-\frac14}| k|^{\frac14}(1+|1+\la|^{\frac12})^{\frac12}.\label{esti: w2L2}
\end{align}
\end{proposition}
We also have the following weighted version, which plays an important role in the proof of space-time estimates in Propositions \ref{Prop: inhomo} and \ref{Prop homo pro}.
\begin{proposition}\label{lemma: Nonslip weighted version}
Let $\rho_k$ be defined by \eqref{def:rhok}. Under the same conditions as in Proposition \ref{Lemma: Nonslip wLinfty wL1}, we have
\begin{align}
&\|\rho^{\frac12}_k w_1\|_{L^2}+\|\rho^{\frac12}_k w_2\|_{L^2}\leq CL^{\frac12},\label{rho12wL2}\\
&\|\rho^{-\frac14}_kw_1\|_{L^2}\leq C\nu^{-\frac{7}{24}}| k|^{\frac{7}{24}}(1+|1-\la|^{\frac12})^{\frac{3}{4}},\label{rho-14w1L2}\\
&\|\rho^{-\frac14}_kw_2\|_{L^2}\leq C\nu^{-\frac{7}{24}}| k|^{\frac{7}{24}}(1+|1+\la|^{\frac12})^{\frac{3}{4}}.\label{rho-14w2L2}
\end{align}
\end{proposition}
The estimate \eqref{rho12wL2} is a direct result of Lemmas \ref{lemma: W1W2} and \ref{lemma: Cij}. The proof of \eqref{rho-14w1L2} and \eqref{rho-14w2L2} is close to that of Proposition 4.4 in \cite{Chen-Li-Wei-Zhang}, with the distinction that we employ $\delta_{*}=\nu^{\frac12}| k|^{-\frac12}(1+|1-\la|^{\frac12})^{-1}$ instead of $\nu^{\frac12}(1+|1-\la|^{\frac12})^{-1}$.

\vspace{0.2cm}
\noindent\textbf{Bounds on $c_1$ and $c_2$.}
We now provide the bounds for $c_1$ and $c_2$ in the intermediate-frequency $10\nu\leq |k|<1$.
\begin{lemma}\label{lemma: c1,c2 FL2}
Let $c_1$ and $c_2$ be defined as in \eqref{def: c1c2}, $\la\in \R$, $F\in L^2(I)$. For $10\nu\leq |k|<1$, it holds that
\begin{align*}
(1+|\la-1|)|c_1|+(1+|\la+1|)|c_2|\leq C\nu^{-\frac16}| k|^{-\frac56}\|F\|_{L^2}.
\end{align*}
\end{lemma}
\begin{proof}
For $|\la-1|\leq 3$, by \eqref{lemma: wL2FL2}, we have
\begin{align*}
|c_1|\leq C\|w_{Na}\|_{L^1}\leq C\nu^{-\frac{1}{6}}| k|^{-\frac56}\|F\|_{L^2}.
\end{align*}

For $|\la-1|\geq3$, we have
\begin{align}\label{la-y}
|\la-y|\geq \frac13|\la-1|, \quad  y\in(-1,1).
\end{align}
Then, we use \eqref{lemma: wL2FL2} and \eqref{1} to get
\begin{align*}
|c_1|=\Big|\int^{1}_{-1}\frac{\sinh  k(1+y)}{\sinh 2 k} w_{Na} dy\Big|
\leq&\Big(\int^1_{-1}\Big(\frac{\sinh  k(1+y)}{(y-\la)\sinh 2 k}\Big)^2dy\Big)^{\frac12}\|(y-\la)w_{Na}\|_{L^2}\\
\leq& (\la-1)^{-1}\|(y-\la)w_{Na}\|_{L^2}\leq | k|^{-1}(\la-1)^{-1}\|F\|_{L^2}\\
\leq &\nu^{-\frac16}| k|^{-\frac56}|\la-1|^{-1}\|F\|_{L^2},
\end{align*}
where we have used the condition $10\nu\leq|k|<1$ in the last inequality.

The estimate of $c_2$ can be obtained by a similar way. Thus, we complete the proof of Lemma \ref{lemma: c1,c2 FL2}.
\end{proof}

\begin{lemma}\label{lemma: c1,c2 FH-1}
Let $c_1$ and $c_2$ be defined as in \eqref{def: c1c2}, $\la\in \R$, $F\in H^{-1}(I)$. For $10\nu\leq |k|<1$, it holds
\begin{align*}
(1+|(\la-1)|)^{\frac34}|c_1|+(1+|(\la+1)|)^{\frac34}|c_2|\leq C\nu^{-\frac12}| k|^{-\frac12}\|F\|_{H^{-1}}.
\end{align*}
\end{lemma}
\begin{proof}
We divide the proof into two cases. 
\vspace{0.1cm}

\textbf{Case 1.} $\boldsymbol{|\la-1|\geq 3}$. Choosing $\delta=\nu^{\frac13}|k|^{-\frac13}$, we have $(-1,1)\cap(\la-\delta,\la+\delta)=\varnothing$. Applying Proposition \ref{lemma: weak resolvent estimate} with $f(y)=\frac{\sinh( k(1+y))}{\sinh(2 k)}$ and $j=1$, we obtain
\begin{equation}\label{c1H-1}
\begin{aligned}
|c_1|=&|\lan w_{Na},f\ran|\\
\leq&C| k|^{-1}\|F\|_{H^{-1}}\Big((|1-\la|+\delta)^{-\frac34}\delta^{-\frac34}+\Big\|\frac{f(y)}{y-\la}\Big\|_{H^1}+\delta^{-1} \Big\|\frac{f(y)}{y-\la}\Big\|_{L^2}\Big).
\end{aligned}
\end{equation}
It follows from Lemma \ref{lemma sinh} that
\begin{align*}
\|f(y)\|_{L^2}\leq \Big\|\frac{\sinh( k(1+y))}{\sinh(2 k)}\Big\|_{L^\infty}\leq C,\quad \|f'(y)\|_{L^2}=\Big\|k\frac{\cosh( k(1+y))}{\sinh(2 k)}\Big\|_{L^2}\leq C,
\end{align*}
which together with \eqref{la-y} yields
\begin{equation}\label{fyy-laL2}
\begin{aligned}
\Big\|\frac{f(y)}{y-\la}\Big\|_{L^2}\leq& \|f\|_{L^2}\Big\|\frac{1}{y-\la}\Big\|_{L^\infty}\lesssim |1-\la|^{-1},\\
\Big\|\frac{f(y)}{y-\la}\Big\|_{H^1}\leq& \|f'\|_{L^2}\Big\|\frac{1}{y-\la}\Big\|_{L^\infty}+\|f\|_{L^\infty}\Big\|\frac{1}{(y-\la)^2}\Big\|_{L^2} +\|f\|_{L^2}\Big\|\frac{1}{y-\la}\Big\|_{L^\infty}\\
\lesssim &|1-\la|^{-1}.
\end{aligned}
\end{equation}
Inserting \eqref{fyy-laL2} into \eqref{c1H-1}, we deduce
\begin{align*}
|c_1|\lesssim &| k|^{-1}\|F\|_{H^{-1}}\big(\delta^{-\frac34}(\delta+|1-\la|)^{-\frac34}+\delta^{-1}(1+|1-\la|)^{-1}\big) \\
\lesssim &\nu^{-\frac12}| k|^{-\frac12}(1+|\la-1|)^{-\frac34}\|F\|_{H^{-1}},
\end{align*}
where we used $|1-\la|^{-1}< 2(1+|1-\la|)^{-1}$.

\textbf{Case 2.} $\boldsymbol{|\la-1|\leq 3}$. Using Proposition \ref{lemma: weak resolvent estimate} with $f(y)=\frac{\sinh( k(1+y))}{\sinh(2 k)}$ and $j=1$, we have
\begin{align*}
|c_1|=&|\lan w_{Na},f\ran|\\
\lesssim &| k|^{-1}\|F\|_{H^{-1}}\Big(\delta^{-\frac32}\|f\|_{L^\infty(E)}+(|1-\la|+\delta)^{-\frac34}\delta^{-\frac34}+\|f\chi\|_{H^1}+\delta^{-1} \|f\chi\|_{L^2}\Big),
\end{align*}
where $E=(-1,1)\cap(\la-\delta, \la+\delta)$.
Thanks to Lemma \ref{lemma sinh},  we obtain
\begin{equation}\label{f esti}
\begin{aligned}
&\|f\|_{L^\infty(E)}\leq C,\\
&\|f\chi\|_{H^1}\leq \|f'\|_{L^2}\|\chi\|_{L^\infty}+\|f\|_{L^\infty}\|\chi'\|_{L^2}+\|f\|_{L^\infty}\|\chi\|_{L^2}
\leq C\delta^{-\frac32},\\
&\|f\chi\|_{L^2}\leq \|f\|_{L^\infty}\|\chi\|_{L^2}\leq C\delta^{-\frac12}.
\end{aligned}
\end{equation}
The condition $0\leq |\la-1|\leq 3$ gives
\begin{align*}
\frac14\leq (1+|\la-1|)^{-1}\leq 1.
\end{align*}
This fact together with \eqref{f esti} implies
\begin{align*}
|c_1|\leq C| k|^{-1}\|F\|_{H^{-1}}\delta^{-\frac32}(1+|(\la-1)|)^{-\frac34}=C\nu^{-\frac12}| k|^{-\frac12}(1+|\la-1|)^{-\frac34}\|F\|_{H^{-1}}.
\end{align*}

Similarly, we can deduce the estimate of $c_2$ and hence complete the proof of Lemma \ref{lemma: c1,c2 FH-1}.
\end{proof}
\begin{proof}[Proof of Proposition \ref{lemma:non-slip boundary,resolvent}]
 Using \eqref{lemma: wL2FL2}, Propositions \ref{Lemma: Nonslip wLinfty wL1}--\ref{lemma: Nonslip weighted version} and Lemma \ref{lemma: c1,c2 FL2}, we obtain
\begin{align*}
&\|\rho^{\frac12}_k w\|_{L^2}\leq \|w_{Na}\|_{L^2}+|c_1|\|\rho^{\frac12}_k w_1\|_{L^2}+|c_2|\|\rho^{\frac12}_k w_2\|_{L^2}\lesssim \nu^{-\frac13}|k|^{-\frac23}\|F\|_{L^2},\\
&\|u\|_{L^2}\leq \|u_{Na}\|_{L^2}+|c_1|\|w_1\|_{L^1}+|c_2|\|w_2\|_{L^1}\lesssim \nu^{-\frac16}|k|^{-\frac56}\|F\|_{L^2},\\
&\|w\|_{L^2}\leq \|w_{Na}\|_{L^2}+|c_1|\|w_1\|_{L^2}+|c_2|\|w_2\|_{L^2}\lesssim  \nu^{-\frac{5}{12}}|k|^{-\frac{7}{12}}\|F\|_{L^2},
\end{align*}
which yields \eqref{wL1+wL2}.

For \eqref{wL2+rhoL2+uL2}, using \eqref{lemma: wL2,FH-1}, \eqref{esti: w1L2}, \eqref{esti: w2L2} and Lemma \ref{lemma: c1,c2 FH-1}, we deduce
\begin{align*}
\|w\|_{L^2}\leq &\|w_{Na}\|_{L^2}+|c_1|\|w_1\|_{L^2}+|c_2|\|w_2\|_{L^2}\\
\lesssim & \nu^{-\frac23}|k|^{-\frac13}\|F\|_{H^{-1}}+\nu^{-\frac{3}{4}}|k|^{-\frac{1}{4}}\|F\|_{H^{-1}}\\
\lesssim & \nu^{-\frac{3}{4}}|k|^{-\frac{1}{4}}\|F\|_{H^{-1}},
\end{align*}
where we have used that $\nu^{\frac{1}{12}}\leq |k|^{\frac{1}{12}}$.
Similarly, by \eqref{lemma: wL2,FH-1}, \eqref{rho12wL2} and Lemma \ref{lemma: c1,c2 FH-1}, we have
\begin{align*}
\|\rho^{\frac12}_k w\|_{L^2}\leq &\|\rho^{\frac12}_kw_{Na}\|_{L^2}+|c_1|\|\rho^{\frac12}_kw_1\|_{L^2}+|c_2|\|\rho^{\frac12}_kw_2\|_{L^2}
\lesssim  \nu^{-\frac{2}{3}}|k|^{-\frac{1}{3}}\|F\|_{H^{-1}}.
\end{align*}
Thanks to \eqref{esti: psiL2omL2}, \eqref{lemma: wL2,FH-1}, \eqref{esti: w1+w2L1} and Lemma \ref{lemma: c1,c2 FH-1}, we obtain
\begin{align*}
\|u\|_{L^2}\leq &\|u_{Na}\|_{L^2}+|c_1|\|w\|_{L^1}+|c_2|\|w_2\|_{L^1}
\lesssim \nu^{-\frac12}|k|^{-\frac12}\|F\|_{H^{-1}}.
\end{align*}
Combining the above three estimates, we complete the proof of \eqref{wL2+rhoL2+uL2}.
\end{proof}
\section{Space-Time Estimates of the Linearized Navier-Stokes Equations}\label{sec: space-time estimate}
In order to derive the nonlinear stability and enhanced dissipation, we first set up the space-time estimates for the linearized Navier-Stokes system
based on the resolvent estimate established in Section \ref{sec: reso esti}. Consider
\begin{align}\label{equ: om}
\pa_t\om+\mathcal{L}\om=-i k f_1-\pa_y f_2, \qquad \om(0)=\om^{in}( k,y),
\end{align}
where $\om=(\pa^2_y - k^2)\psi$ with $\psi(\pm 1)=\psi'(\pm 1)=0$. 
Denote
\begin{align*}
\|f\|_{L^p H^s}=\bn{\|f(t,y)\|_{H^s_y(I)}}_{L^p_t(\R^{+})}, \quad \|f\|_{L^pL^q}=\bn{\|f(t,y)\|_{L^q_y(I)}}_{L^p_t(\R^{+})}.
\end{align*}

The space-time estimates of the solution to \eqref{equ: om} will be established separately on the following three intervals:
\begin{itemize}
\item $\{k: 0\leq |k|<10\nu\}$,
\item $\{k: 10\nu\leq |k|<1\}$,
\item $\{k:  |k|\geq 1\}$.
\end{itemize}

For $|k|\geq 1$, the space-time estimates of the solution to \eqref{equ: om} are proved in \cite{Chen-Li-Wei-Zhang} as follows:
\begin{equation}
\begin{aligned}\label{Th nuxi2leq1 k>1}
&\|(1-|y|)^{\frac12}\om\|^2_{L^\infty L^2}+(\nu k^2)^{\frac12}\|\om\|^2_{L^2L^2}+|k|^2\|u\|^2_{L^2 L^2}+| k|\|u\|^2_{L^\infty L^\infty}\\
\leq& C\min\{\nu^{-\frac13}|k|^{\frac43},\nu^{-1}\}\|f_1\|^2_{L^2 L^2}+C\nu^{-1}\|f_2\|^2_{L^2 L^2}+\|\om^{in}\|^2_{L^2}+|k|^{-2}\|\pa_y\om^{in}\|^2_{L^2},
\end{aligned}
\end{equation}
where $u=(\pa_y,-ik)(\pa^2_y-k^2)^{-1}\om$ and $f_1, f_2\in L^2 L^2$.
\vspace{0.1cm}

Now, we address the remaining cases in Subsections \ref{subsec: 0-10nu} and \ref{subsec: 10nu-1}.
\subsection{Space-time estimates of low frequencies $0\leq |k|<10\nu$}\label{subsec: 0-10nu}
We shall utilize energy method to obtain the space-time estimates in frequency range $0\leq |k|<10\nu$. Decomposing the stream function and using the optimal Sobolev constant from Wirtinger's inequality, we obtain the maximal frequency range $0\leq |k|<\frac{\pi^3}{3}\nu$. In this case, we gain the the space-time estimates uniformly with respect to $k$.
\begin{theorem}\label{Th xi small}
Let $0\leq| k|< 10\nu$ and $\om$ be a solution of \eqref{equ: om} with $\om^{in}\in L^2(I)$ and $f_1, f_2\in L^2 L^2$. Then there exists a constant $C>0$ independent of $\nu,  k$ such that
\begin{align*}
&\|\om\|^2_{L^\infty L^2}+\nu\|u\|^2_{L^2 L^2}+\|u\|^2_{L^\infty L^\infty}+\nu\|\om\|^2_{L^2L^2}\\
\leq &C\nu^{-\frac13}| k|^{\frac43}\|f_1\|^2_{L^2 L^2}+C\nu^{-1}\|f_2\|^2_{L^2 L^2}+\|\om^{in}\|^2_{L^2},
\end{align*}
where $u=(\pa_y,-ik)(\pa^2_y-k^2)^{-1}\om$.
\end{theorem}
\begin{proof}
For $k=0$, from \cite{Chen-Li-Wei-Zhang}, we obtain
\begin{align*}
\|\om\|^2_{L^\infty L^2}+\nu\|u\|^2_{L^2 L^2}+\|u\|^2_{L^\infty L^\infty}+\nu\|\om\|^2_{L^2L^2}
\leq C\nu^{-1}\|f_2\|^2_{L^2 L^2}+\|\om^{in}\|^2_{L^2}.
\end{align*}
For $0<|k|<10\nu$, we divide the proof into two steps.
\vspace{0.1cm}

\textbf{Step 1.} The estimates of $\|u\|^2_{L^2 L^2}$ and $\|\om\|^2_{L^2 L^2}$.
Taking $L^2$ inner product of \eqref{equ: om}, we get
\begin{align}\label{inte psi}
\lan (\pa_t-\nu(\pa^2_y- k^2)+i k y)\om, -\psi \ran=\lan-i k f_1-\pa_y f_2,-\psi\ran.
\end{align}
Integrating by parts, we have
\begin{align*}
&\lan \pa_t u,u\ran+\nu\|\om\|^2_{L^2}+i k \int^1_{-1}\psi'\overline{\psi} dy +i k \int^1_{-1}y|\psi|^2 dy+i k^3\int^1_{-1}y|\psi|^2 dy\\
=&\lan -i k f_1-\pa_y f_2,-\psi\ran=\lan i k f_1,\psi\ran-\lan f_2,\pa_y \psi\ran.
\end{align*}
It follows by taking the real part of the above equality that
\begin{equation}\label{esti: uL2 nusmall}
\begin{aligned}
    \frac12\frac{d}{dt}\|u\|^2_{L^2}+\nu\|\om\|^2_{L^2}=-\operatorname{Re}\left(i k \int^1_{-1}\psi'\overline{\psi} dy\right)+\operatorname{Re}(\lan i k f_1,\psi\ran-\lan f_2,\pa_y \psi\ran).
\end{aligned}
\end{equation}

We first deal with the term $-\operatorname{Re}\left(i k \int^1_{-1}\psi'\overline{\psi} dy\right)$. Let
\begin{align}\label{basis}
\varphi_j(y)=\sin(\pi j(y+1)/2),\quad j\in \Z_{+}.
\end{align}
$\{\varphi_j\}_{j\geq 1}$ forms an orthogonal basis in $L^2([-1,1])$, and so we have
$$\psi=\sum_{j=1}^{+\infty} \lan \varphi_j, \psi\ran\varphi_j=:\sum_{j=1}^{+\infty}a_j \varphi_j.$$
We decompose $\int^1_{-1}\psi'\overline{\psi} dy$ into
\begin{equation}\label{intpsipsi'}
\begin{aligned}
    \int_{-1}^{1}\psi' \overline{\psi} dy=&\int_{-1}^{1} \Big(a_1\varphi_1+\sum_{j\geq 2}a_j\varphi_j\Big)' \Big(\overline{a}_1\varphi_1+\sum_{j\geq 2}\overline{a}_j\varphi_j\Big) dy\\
    =&\int_{-1}^{1}\Big(\sum_{j\geq 2}a_j\varphi_j\Big)'\sum_{j\geq 2}\overline{a}_j\varphi_j
    +(a_1\varphi_1)' \sum_{j\geq 2}\overline{a}_j\varphi_j+
    \Big(\sum_{j\geq 2}a_j\varphi_j\Big)' \overline{a}_1\varphi_1dy,
\end{aligned}
\end{equation}
where we have used that
\begin{align*}
\int^1_{-1}(a_1\varphi_1)' \overline{a}_1\varphi_1 dy =\int^1_{-1}|a_1|^2 \sin(\pi (y+1)/2)\cos(\pi (y+1)/2)dy=0.
\end{align*}
Due to
\begin{align*}
\psi''=\sum_{j=1}^{+\infty}\Big\lan \psi'', \sin\frac{\pi j(y+1)}{2}\Big\ran \sin\frac{\pi j(y+1)}{2}=:\sum_{j=1}^{+\infty}\psi''_j,
\end{align*}
it follows from the boundary condition $\psi(\pm1)=\psi'(\pm1)=0$ that
\begin{align*}
\Big\lan \psi'', \sin\frac{\pi j(y+1)}{2}\Big\ran=&\psi'\sin\frac{\pi j(y+1)}{2}\Big|^{1}_{-1}-\frac{\pi j}{2}\Big\lan \psi', \cos\frac{\pi j(y+1)}{2}\Big\ran\\
=& -\frac{\pi j}{2}\psi\cos\frac{\pi j(y+1)}{2}\Big|^{1}_{-1}-\Big(\frac{\pi j}{2}\Big)^2\Big\lan \psi, \sin\frac{\pi j(y+1)}{2}\Big\ran\\
=& -\Big(\frac{\pi j}{2}\Big)^2\Big\lan \psi, \sin\frac{\pi j(y+1)}{2}\Big\ran.
\end{align*}
This equality implies
\begin{align*}
\psi''=-\sum^{+\infty}_{j=1}\Big(\frac{\pi j}{2}\Big)^2a_j\varphi_j ,\quad \Big\|\sum_{j\geq 2}\psi''_j\Big\|_{L^2}=\Big(\sum_{j\geq 2}^{+\infty}\Big|\Big(\frac{\pi j}{2}\Big)^2a_j \Big|^2\Big)^{\frac{1}{2}},\quad \|\psi''_1\|_{L^2}=\frac{\pi^2}{4}|a_1|.
\end{align*}
Using H\"older's inequality, we have
\begin{equation}\label{esti: Pgeq2P1}
\begin{aligned}
   \int_{-1}^{1} \psi_1' \sum_{j\geq 2}\overline{\psi}_j dy\leq&  \Big\|\frac{\pi}{2}a_1 \cos\frac{\pi(y+1)}{2}\Big\|_{L^2} \Big\|\sum_{j\geq 2}\overline{a}_j \varphi_j\Big\|_{L^2}= \frac{\pi}{2}|a_1 |\Big(\sum_{j\geq 2}|a_j|^2\Big)^{\frac12}\\
\leq &   \frac{1}{4}\left(\frac{2}{\pi}\right)^{3}\|(\psi_1)''\|_{L^2}
   \Big\|\Big(\sum_{j\geq 2}\psi_j\Big)''\Big\|_{L^2}\\
   \leq&\frac{1}{4}\left(\frac{2}{\pi}\right)^{3}\left(\frac{1}{2}\|\psi''_1\|_{L^2}^{2}+\frac{1}{2}\Big\|\sum_{j\geq 2}\psi''_j\Big\|_{L^2}^{2}\right)
   =\frac{1}{\pi^{3}}\|\psi''\|^2_{L^2}.
\end{aligned}
\end{equation}
Integrating by parts, we directly get
\begin{equation}\label{esti: P1Pgeq2}
\begin{aligned}
   \int_{-1}^{1} \Big(\sum_{j\geq 2}\psi_j\Big)' \overline{\psi}_1dy
   = - \int_{-1}^{1} \sum_{j\geq 2}\psi_j(\overline{\psi}_1)'dy
   \leq\frac{1}{\pi^{3}}   \|\psi''\|^2_{L^2}.
\end{aligned}
\end{equation}
Similarly, we derive
\begin{equation}
\begin{aligned}\label{esti: Pgeq2geq2}
   \int_{-1}^{1} \Big(\sum_{j\geq 2}\psi_j\Big)' \sum_{j\geq 2}\overline{\psi}_jdy\leq& \Big\|\Big(\sum_{j\geq 2}a_j \varphi_j\Big)'\Big\|_{L^2} \Big\|\sum_{j\geq 2}\overline{a}_j \varphi_j\Big\|_{L^2} \\
   =&\Big\|\sum_{j\geq 2}a_j \frac{\pi j}{2}\cos(\frac{\pi j(y+1)}{2})\Big\|_{L^2}\Big(\sum_{j\geq 2}|a_j|^2\Big)^{\frac12}\\
   \leq&\Big(\sum_{j\geq 2}|\frac{\pi j}{2}a_j|^2\Big)^{\frac12}\Big(\sum_{j\geq 2}|a_j|^2\Big)^{\frac12}\leq \frac{1}{\pi^{3}}
   \|P^{y}_{j\geq 2}(\psi'')\|^2_{L^2}.
\end{aligned}
\end{equation}
Inserting \eqref{esti: Pgeq2P1}, \eqref{esti: P1Pgeq2} and \eqref{esti: Pgeq2geq2} into \eqref{intpsipsi'}, we obtain
\begin{align*}
    \int_{-1}^{1}\psi' \overline{\psi} dy\leq \frac{3}{\pi^3}\|\psi''\|^2_{L^2}.
\end{align*}
This inequality together with $0<| k|< 10\nu$ and \eqref{psi2L2} implies
\begin{align}\label{ineq: intpsi}
\left|\operatorname{Re}\left(i k \int^1_{-1}\psi'\overline{\psi} dy\right)\right|\leq \frac{3}{\pi^3}| k|\|\psi''\|^2_{L^2}\leq \frac{30}{\pi^3}\nu\|\om\|^2_{L^2}.
\end{align}

For the term $\operatorname{Re}(\lan i k f_1,\psi\ran-\lan f_2,\pa_y \psi\ran)$, we use $0< | k|< 10\nu$ and \eqref{esti: psiL2omL2} to get
\begin{equation}
\begin{aligned}\label{ineq: psiL2 xismall}
&|\operatorname{Re}(\lan i k f_1,\psi\ran-\lan f_2,\pa_y \psi\ran)|\\
\leq &C\nu^{-\frac13}| k|^{\frac43}\|f_1\|^2_{L^2}+C\nu^{-1}\|f_2\|^2_{L^2} + 10^{-3}\nu^{\frac13} | k|^{\frac23}\|\psi\|^2_{L^2}+10^{-3}\nu\|\psi'\|^2_{L^2}\\
\leq &C\nu^{-\frac13}| k|^{\frac43}\|f_1\|^2_{L^2}+C\nu^{-1}\|f_2\|^2_{L^2}+0.02\nu \|\om\|^2_{L^2}.
\end{aligned}
\end{equation}

Inserting \eqref{ineq: intpsi} and \eqref{ineq: psiL2 xismall} into \eqref{esti: uL2 nusmall}, we deduce
\begin{align}\label{dtuL2}
\frac{d}{dt}\|u\|^2_{L^2}+\nu\|\om\|^2_{L^2}\lesssim \nu^{-\frac13}| k|^{\frac43}\|f_1\|^2_{L^2}+\nu^{-1}\|f_2\|^2_{L^2}.
\end{align}
Integrating with respect to $t$, we further obtain
\begin{align}\label{nuxi2geq1 uL2omL2}
\|u\|^2_{L^\infty L^2}+\nu \|\om\|^2_{L^2L^2}\leq C\nu^{-\frac13}| k|^{\frac43}\|f_1\|^2_{L^2L^2}+C\nu^{-1}\|f_2\|^2_{L^2L^2}+\|u^{in}\|^2_{L^2}.
\end{align}
Due to \eqref{esti: psiL2omL2},
we arrive at
\begin{align}\label{nuxi2geq1 uLtinftyuLt2}
\nu\|u\|^2_{L^2L^2}+\nu\|\om\|^2_{L^2L^2}\leq C\nu^{-\frac13}| k|^{\frac43}\|f_1\|^2_{L^2L^2}+C\nu^{-1}\|f_2\|^2_{L^2L^2}+\|u^{in}\|^2_{L^2}.
\end{align}

\textbf{Step 2.} The estimates of $\|\om\|_{L^\infty L^2}$ and $\|u\|_{L^\infty L^\infty}$.
Let $F_1=\pa_t \psi+i k y \psi$. It satisfies
\begin{align*}
(\pa^2_y- k^2)F_1 =\pa_t \om+iky\om+2i k \pa_y \psi,\quad F_1(\pm 1)=\pa_y F_1(\pm 1)=0.
\end{align*}
Integrating by parts, we have
\begin{align*}
&\lan i k f_1, F_1\ran-\lan f_2,\pa_y F_1\ran=\lan -i k f_1-\pa_y f_2, -F_1\ran \\
=&\lan (\pa_t-\nu(\pa^2_y- k^2)+i k y)\om,-F_1\ran\\
=&\lan (\pa^2_y- k^2)F_1-2i k\pa_y \psi-\nu(\pa^2_y- k^2)\om,-F_1\ran\\
=&\|\pa_y F_1\|^2_{L^2}+| k|^2\|F_1\|_{L^2}+\lan 2i k\pa_y \psi, F_1\ran+\nu \lan\om,\pa_t\om+i k y \om+2i k\pa_y \psi \ran.
\end{align*}
Taking the real part and then using the Cauchy-Schwarz inequality, we obtain
\begin{equation}\label{dtomL2+FL2}
\begin{aligned}
&\frac{\nu}{2}\frac{d}{dt}\|\om\|^2_{L^2}+\|\pa_y F_1\|^2_{L^2}+| k|^2\|F_1\|^2_{L^2}\\
\leq& 2\nu | k||\lan \om, \pa_y \psi\ran|+2| k||\lan \pa_y \psi, F_1\ran|+|\lan i k f_1, F_1\ran|+|\lan f_2,\pa_y F_1\ran|\\
\leq & \nu | k|\|\om\|^2_{L^2}+\nu  k\|\pa_y \psi\|^2_{L^2}+4| k|^2\|\pa_y\psi\|^2_{L^2}+\frac{1}{4}\|F_1\|^2_{L^2}\\
&+| k|^2\|f_1\|^2_{L^2}+\frac{1}{4}\|F_1\|^2_{L^2}+\frac12\|f_2\|^2_{L^2}+\frac12\|\pa_y F_1\|^2_{L^2}.
\end{aligned}
\end{equation}
Noting that $F_1(\pm1)=0$, we get by Wirtinger's inequality
\begin{align*}
\|F_1\|^2_{L^2}\leq \frac{4}{\pi^2}\|\pa_y F_1\|^2_{L^2}.
\end{align*}
Inserting this inequality into \eqref{dtomL2+FL2} and using \eqref{esti: psiL2omL2}, we deduce
\begin{align*}
\nu \frac{d}{dt}\|\om\|^2_{L^2}\leq | k|^2 \|f_1\|^2_{L^2}+\|f_2\|^2_{L^2}+32(\nu k+| k|^2)\|\om\|^2_{L^2},
\end{align*}
which, together with $0\leq |k|<10\nu$, further implies
\begin{align*}
\|\om\|^2_{L^\infty L^2}\lesssim \nu^{-1}| k|^2\|f_1\|^2_{L^2L^2}+\nu^{-1}\|f_2\|^2_{L^2 L^2}+C\nu \|\om\|^2_{L^2L^2}+ \|\om^{in}\|^2_{L^2}.
\end{align*}
By appealing to \eqref{nuxi2geq1 uL2omL2}, we immediately get
\begin{align*}
\|\om\|^2_{L^\infty L^2}\lesssim& \nu^{-\frac13}| k|^{\frac43}\|f_1\|^2_{L^2L^2}+\nu^{-1}\|f_2\|^2_{L^2 L^2}+ \|\om^{in}\|^2_{L^2}.
\end{align*}
Using \eqref{esti: psiLinftywL2}, we have
\begin{align*}
\|\om\|^2_{L^\infty L^2}+\|u\|^2_{L^\infty L^\infty}\lesssim& \nu^{-\frac13}| k|^{\frac43}\|f_1\|^2_{L^2L^2}+\nu^{-1}\|f_2\|^2_{L^2 L^2}+ \|\om^{in}\|^2_{L^2}.
\end{align*}
This together with \eqref{nuxi2geq1 uLtinftyuLt2} completes the proof of Theorem \ref{Th xi small}.
\end{proof}
\subsection{Space-time estimates of intermediate frequencies $10\nu\leq |k|<1$}\label{subsec: 10nu-1}
Compared to the case $|k|\geq1$,  in the energy functional, the terms  including $\|u\|_{L^\infty L^\infty}$, $\|u\|_{L^2 L^2}$,$\|\om\|_{L^2 L^2}$  and the  term  $\|\om\|_{L^\infty L^2}$ eliminate factor $|k|^{\frac12}$ and factor $|k|^{\frac34}$, respectively,  so it leads to the space-time estimates more available for $10\nu\leq |k|<1$.
\begin{theorem}\label{Th: linear problem}
Let $10\nu\leq |k|<1$ and $\om$ be a solution of \eqref{equ: om} with $\om^{in}\in H^1$ and $f_1, f_2 \in L^2L^2$. It holds that
\begin{align*}
&\|(1-|y|)^{1/2}\om\|_{L^\infty L^2}+\nu^{\frac12}|k|^{-\frac12}\|\om\|^2_{L^\infty L^2}+\|u\|^2_{L^\infty L^\infty}+| k|\|u\|^2_{L^2 L^2}+\nu^{\frac12}  k^{\frac12}\|\om\|^2_{L^2 L^2}\\
\lesssim &\|\om^{in}\|^2_{L^2}+\nu^{\frac23}| k|^{-\frac23}\|\pa_y\om^{in}\|^2_{L^2}+\nu^{-\frac13}|k|^{\frac43}\|f_1\|^2_{L^2L^2}+\nu^{-1}\|f_2\|^2_{L^2L^2},
\end{align*}
where $u=(\pa_y,-ik)(\pa^2_y-k^2)^{-1}\om$.
\end{theorem}
Roughly speaking, we can seek a cut-off function $\varrho$ to substitute $(1-|y|)^{\frac12}$ such that $\operatorname{supp}(1-\varrho)$ is small since $(1-|y|)^{\frac12}\om$ and $\om$ are equivalent in regions far from the boundary $\pm1$. In fact, it suffices to ensure the measure of the region $\{y:\, \varrho\leq (1-|y|)^{\frac12}$\} does not exceed the coefficient of $\|\om\|_{L^\infty L^2}$. A similar analysis as in \cite[P167]{Chen-Li-Wei-Zhang}, we obtain
\begin{align*}
&\|(1-|y|)^{\frac12}\om\|^2_{L^\infty L^2}\leq \nu^{\frac12}|k|^{-\frac12}\|\om\|^2_{L^\infty L^2}+\|\rho^{\frac32}_k\om\|^2_{L^\infty L^2},\\
&\|u\|^2_{L^\infty L^\infty}\lesssim \nu^{\frac12}|k|^{-\frac12}\|\om\|^2_{L^\infty L^2}+\|\rho^{\frac32}_k\om\|^2_{L^\infty L^2}.
\end{align*}
Due to $L=(k/\nu)^{\frac13}>1$ for $10\nu\leq |k|<1$, by a similar argument as in \cite[Proposition 6.7]{Chen-Li-Wei-Zhang}, we deduce
\begin{align*}
\|\rho^{\frac32}_k\om\|^2_{L^\infty L^2}\leq &\|\om^{in}\|^2_{L^2}+\nu^{\frac23}| k|^{-\frac23}\|\pa_y \om^{in}\|^2_{L^2}\\
&+\nu^{\frac13} | k|^{\frac23}\|\rho^{\frac12}_k\om\|^2_{L^2 L^2}
+\nu^{-\frac13}| k|^{\frac43}\|f_1\|^2_{L^2L^2} +\nu^{-1}\|f_2\|^2_{L^2L^2}.
\end{align*}
Therefore, the proof of Theorem \ref{Th: linear problem} can be reduce to estimate
\begin{equation}\label{full problem L2L2}
\begin{aligned}
&\nu^{\frac13} | k|^{\frac23}\|\rho^{\frac12}_k \om\|^2_{L^2L^2}+| k|\|u\|^2_{L^2L^2} +\nu^{\frac12} | k|^{\frac12}\|\om\|^2_{L^2 L^2}+\nu^{\frac12} | k|^{-\frac12}\|\om\|^2_{L^\infty L^2}\\
\lesssim &\|\om^{in}\|^2_{L^2}+\nu^{\frac23}| k|^{-\frac23}\|\pa_y \om^{in}\|^2_{L^2}+\nu^{-\frac13}| k|^{\frac43}\|f_1\|^2_{L^2L^2} +\nu^{-1}\|f_2\|^2_{L^2L^2}.
\end{aligned}
\end{equation}

In order to estimate \eqref{full problem L2L2}, we decompose the solution of \eqref{equ: om} as $\om=\om_{I}+\om_{H}$, where $\om_{I}$ solves
\begin{equation}\label{equ: omI}
\left\{
\begin{aligned}
&(\pa_t-\nu(\pa^2_y- k^2)+i k y)\om_{I}=-i k f_1-\pa_y f_2, \\
&\om_{I}=(\pa^2_y- k^2)\psi_{I},\quad \psi_{I}(\pm1)=\psi'_{I}(\pm1)=0,\\
&\om_{I}(0, k,y)=0,
\end{aligned}
\right.
\end{equation}
and $\om_{H}$ solves
\begin{equation}\label{equ: omH}
\left\{
\begin{aligned}
&(\pa_t-\nu(\pa^2_y- k^2)+i k y)\om_{H}=0,\\
&\om_{H}=(\pa^2_y- k^2)\psi_{H},\quad \psi_{H}(\pm1)=\psi'_{H}(\pm1)=0,\\
&\om_{H}(0, k,y)=\om^{in}( k,y).
\end{aligned}
\right.
\end{equation}

We consider the equation \eqref{equ: omI} in Fourier space
\begin{equation*}\left\{
\begin{aligned}
&-\nu (\pa^2_y- k^2)w_{I}(\la, k,y)+i k (y+\la/ k)w_{I}(\la, k,y)=-i k F_1(\la, k,y)-\pa_y F_2(\la, k,y),\\
&\lan w_I(\la, k,y), e^{\pm  k y}\ran =0,\qquad \int^{+\infty}_{-\infty}w_I(\la,k,y)dy=0.
\end{aligned}
\right.
\end{equation*}
where
\begin{align*}
&w_I(\la, k,y)=:\int^{+\infty}_0 \om_{I}(t, k,y) e^{-it\la} dt,\quad F_{j}(\la, k,y)=:\int^{+\infty}_0 f_{j}(t, k,y)e^{-it\la} dt,\quad  j=1,2.
\end{align*}

Let $\phi_1$, $\phi_2$ solve \eqref{equ: psi1} and \eqref{equ: psi2} with $\epsilon=0$, $\la$ replaced by $\la'=-\la/ k$. We define $w_j=(\pa^2_y- k^2)\phi_j\, (j=1,2)$.
Then, $w_{I}$ can be decomposed into
\begin{align*}
w_{I}(\la, k,y)=w^{(1)}_{Na}+w^{(2)}_{Na}+(c^{(1)}_1(\la)+c^{(2)}_1(\la))w_1+(c^{(1)}_2(\la)+c^{(2)}_2(\la))w_2,
\end{align*}
where
\begin{align}
&(i\la-\nu(\pa^2_y- k^2)+i k y)w^{(1)}_{Na}(\la, k,y)=-i k F_1(\la, k,y),\,\, w_{Na}^{(1)}(\pm1)=0, \label{equ: w1Na}\\
&(i\la-\nu(\pa^2_y- k^2)+i k y)w^{(2)}_{Na}(\la, k,y)=-\pa_y F_2(\la, k,y),\,\, w_{Na}^{(2)}(\pm1)=0, \label{equ: w2Na}
\end{align}
and
\begin{align*}
c^{(j)}_l(\la)=&-\int^1_{-1}\frac{\sinh  k(1+(-1)^{l+1}y)}{\sinh 2 k} w^{(j)}_{Na}(\la, k,y) dy, \quad l=1,2.
\end{align*}
Let $\om_{Na}$ be the solution of
\begin{equation}\left\{
\begin{aligned}\label{equ:omNa}
&(\pa_t-\nu (\pa^2_y- k^2)+i  k y)\om_{Na}=-i k f_1-\pa_y f_2, \\
&\om_{Na}\big|_{t=0}=0, \,\,\, \om_{Na}\big|_{y=\pm1}=0.
\end{aligned}
\right.
\end{equation}
Thus, we can rewrite
\begin{align*}
\om_{I}=\om_{Na}+\om^{(1)}_1+\om^{(2)}_1+\om^{(1)}_2+\om^{(2)}_2,
\end{align*}
where
\begin{align*}
\om^{(j)}_l(t, k,y)=\frac{1}{2\pi}\int_{\R}c^{(j)}_l(\la)w_l(\la, k,y) e^{it\la} d\la,\quad t>0, \quad j,l\in\{1,2\}.
\end{align*}
Now, we first give the space-time estimates of $\om_{I}$.
\begin{proposition}\label{Prop: inhomo}
Let $10\nu\leq |k|<1$. Assume that $\rho_k$ is defined by \eqref{def:rhok} and $\om_{I}$ is a solution of \eqref{equ: omI} with $f_1, f_2\in L^2L^2$. Then there exists a constant $C>0$ independent of $\nu$ and $ k$, such that
\begin{equation}\label{esti: inhomo}
\begin{aligned}
&\nu^{\frac13}| k|^{\frac23}\|\rho^{\frac12}_k\om_{I}\|^2_{L^2 L^2}+| k|\|u_{I}\|^2_{L^2 L^2}+\nu^{\frac12} | k|^{\frac12}\|\om_{I}\|^2_{L^2L^2}+\nu^{\frac12} | k|^{-\frac12}\|\om_{I}\|^2_{L^\infty L^2}\\
\leq& C(\nu^{-\frac13}| k|^{\frac43}\|f_1\|^2_{L^2 L^2}+\nu^{-1}\|f_2\|^2_{L^2 L^2}).
\end{aligned}
\end{equation}
\end{proposition}
\begin{proof}
Define
$$w^{(1)}_I=:w^{(1)}_{Na}+c^{(1)}_1(\la)w_1+c^{(1)}_2(\la)w_2,\quad w^{(2)}_I=:w^{(2)}_{Na}+c^{(2)}_1(\la)w_1+c^{(2)}_2(\la)w_2,$$
which satisfy
\begin{align*}
-\nu (\pa^2_y- k^2)w^{(1)}_I+i k (y+\la/ k)w^{(1)}_I(\la, k,y)=-i k F_1(\la, k,y), \quad \lan w^{(1)}_{I}(\la, k,y), e^{\pm  k y}\ran=0,\\
-\nu (\pa^2_y- k^2)w^{(2)}_I+i k (y+\la/ k)w^{(2)}_I(\la, k,y)=-\pa_y F_2(\la, k,y),\quad \lan w^{(2)}_{I}(\la, k,y), e^{\pm  k y}\ran=0.
\end{align*}
Let $u^{(1)}_{I}=(\pa_y,ik)(\pa_y^2-k^2)w^{(1)}_{I}$, $u^{(2)}_{I}=(\pa_y,ik)(\pa_y^2-k^2)w^{(2)}_{I}$.
By Plancherel's formula and Proposition \ref{lemma:non-slip boundary,resolvent}, we deduce
\begin{align*}
&\nu^{\frac13} | k|^{\frac23}\|\rho^{\frac12}_{k}\om_{I}\|^2_{L^2L^2}+| k|\|u_{I}\|^2_{L^2L^2}+\nu^{\frac12} | k|^{\frac12}\|\om_{I}\|^2_{L^2L^2}\\
\sim & \nu^{\frac13}|k|^{\frac23}\big\|\|\rho^{\frac12}_{k}w^{(1)}_{I}\|^2_{L^2_y}\big\|^2_{L^2_{\la}}+| k|\big\|\|u^{(1)}_{I}\|^2_{L^2_y}\big\|^2_{L^2_{\la}} +\nu^{\frac12} | k|^{\frac12}\big\|\|w^{(1)}_{I}\|^2_{L^2_y}\big\|^2_{L^2_{\la}}\\
&+ \nu^{\frac13}|k|^{\frac23}\big\|\|\rho^{\frac12}_{k}w^{(2)}_{I}\|^2_{L^2_y}\big\|^2_{L^2_{\la}}+| k|\big\|\|u^{(2)}_{I}\|^2_{L^2_y}\big\|^2_{L^2_{\la}} +\nu^{\frac12} | k|^{\frac12}\big\|\|w^{(2)}_{I}\|^2_{L^2_y}\big\|^2_{L^2_{\la}}\\
\lesssim&\nu^{-\frac13}|k|^{-\frac23}\|kF_1\|^2_{L^2_{\la}L^2_y}+\nu^{-1}\|F_2\|^2_{L^2_{\la}L^2_y}\\
\sim&\nu^{-\frac13}| k|^{\frac43}\|f_1\|^2_{L^2L^2}+\nu^{-1}\|f_2\|^2_{L^2L^2}.
\end{align*}

It remains to bound $\|\om_{I}\|_{L^\infty L^2}$. By Plancherel's theorem and Proposition \ref{pro: reso esti FL2}, we obtain
\begin{equation}\label{esti: wNaL2L2 f1f2}
\begin{aligned}
(\nu  k^2)^{\frac13}\|\om_{Na}\|^2_{L^2 L^2}\leq& (\nu k^2)^{\frac13}\big(\big\|\|w^{(1)}_{Na}(\la)\|^2_{L^2_y}\big\|^2_{L^2_{\la}} +\big\|\|w^{(2)}_{Na}(\la)\|^2_{L^2_y}\big\|^2_{L^2_{\la}}\big)\\
\lesssim & \nu^{-\frac13}| k|^{\frac43}\|F_1(\la)\|^2_{L^2_{\la}L^2_y}+\nu^{-1}\|F_2(\la)\|^2_{L^2_{\la}L^2_y}\\
\sim & \nu^{-\frac13}| k|^{\frac43}\|f_1\|^2_{L^2L^2}+\nu^{-1}\|f_2\|^2_{L^2 L^2}.
\end{aligned}
\end{equation}
For $\|\om_{Na}\|_{L^\infty L^2}$, we use the energy method for \eqref{equ:omNa} to deduce
\begin{align*}
&\pa_t\|\om_{Na}\|^2_{L^2}+\nu \|\pa_y \om_{Na}\|^2_{L^2}+2\nu  k^2\|\om_{Na}\|^2_{L^2}\\
\leq &\nu^{-\frac13}| k|^{\frac43}\|f_1\|^2_{L^2}+\nu^{-1}\|f_2\|^2_{L^2}+(\nu k^2)^{\frac13}\|\om_{Na}\|^2_{L^2}.
\end{align*}
Integrating with respect to time, we further use \eqref{esti: wNaL2L2 f1f2} to obtain
\begin{equation}\label{esti: omNaL2f1f2}
\begin{aligned}
\|\om_{Na}\|^2_{L^\infty L^2}\lesssim & \nu^{-\frac13}| k|^{\frac43}\|f_1\|^2_{L^2L^2}+\nu^{-1}\|f_2\|^2_{L^2L^2}+(\nu k^2)^{\frac13}\|\om_{Na}\|^2_{L^2L^2} \\
\lesssim &\nu^{-\frac13}| k|^{\frac43}\|f_1\|^2_{L^2L^2}+\nu^{-1}\|f_2\|^2_{L^2L^2}.
\end{aligned}
\end{equation}

Let $\la'=-\la/ k$. The equation \eqref{equ: w1Na} can be rewritten as
\begin{align*}
(-\nu (\pa^2_y- k^2)+i k (y-\la'))w^{(1)}_{Na}(\la, k,y)=-i k F_1(\la, k,y), \quad w^{(1)}_{Na}|_{y=\pm 1}=0.
\end{align*}
By Proposition \ref{Lemma: Nonslip wLinfty wL1} and Lemma \ref{lemma: c1,c2 FL2}, we conclude that
\begin{align*}
(1+|(-\la/ k-1)|)^{\frac34}|c^{(1)}_1(\la)|\|w_1\|_{L^2}\leq C\nu^{-\frac{5}{12}}| k|^{-\frac{7}{12}}\| k F_1(\la)\|_{L^2}.
\end{align*}
In light of $\|(1+|\la/ k+1|)^{-\frac34}\|_{L^2}\leq C| k|^{\frac12}$, we have
\begin{align*}
\|\om^{(1)}_1(t)\|_{L^2_y}\leq &\frac{1}{2\pi}\int_{\R}|c^{(1)}_{1}(\la)|\|w_1(\la)\|_{L^2_y}d\la\\
\leq &C\big\|(1+|\la/ k+1|)^{-\frac34}\nu^{-\frac{5}{12}}| k|^{-\frac{7}{12}}\| k F_1(\la)\|_{L^2}\big\|_{L^1_{\la}}\\
\leq & C\nu^{-\frac{5}{12}}| k|^{\frac{5}{12}}\|(1+|\la/ k+1|)^{-\frac34}\|_{L^2_{\la}}\|F_1(\la)\|_{L^2_{\la}L^2_y}\\
\leq& C\nu^{-\frac{5}{12}}| k|^{\frac{11}{12}}\|f_1\|_{L^2L^2}.
\end{align*}
This inequality yields
\begin{align}\label{om11L2f1L2}
\nu^{\frac14} | k|^{-\frac14}\|\om^{(1)}_1\|_{L^\infty L^2}\leq C\nu^{-\frac16}| k|^{\frac23}\|f_1\|_{L^2L^2}.
\end{align}
Similarly, we obtain
\begin{align}\label{omom12L2f1L2}
\nu^{\frac14} | k|^{-\frac14}\|\om^{(1)}_2\|_{L^\infty L^2}\leq C\nu^{-\frac16}| k|^{\frac23}\|f_1\|_{L^2L^2}.
\end{align}

By Proposition \ref{lemma: Nonslip weighted version} and Lemma \ref{lemma: c1,c2 FH-1}, we have
\begin{align*}
&(1+|(-\la/ k-1)|)^{\frac34}|c^{(2)}_1(\la)|\|\rho_k^{\frac12} w_{1}\|_{L^2_y}\leq C\nu^{-\frac23}| k|^{-\frac13}\|F_2(\la)\|_{L^2_{y}},\\
&(1+|(-\la/ k-1)|)^{\frac38}|c^{(2)}_1(\la)|\|\rho_k^{-\frac14} w_{1}\|_{L^2_y}\leq C\nu^{-\frac{19}{24}}| k|^{-\frac{5}{24}}\|F_2(\la)\|_{L^2_y},
\end{align*}
which give
\begin{align}\label{c21rhow1+c21rho-14w1L2}
|\la+ k|^{\frac32}|c^{(2)}_1(\la)|^2\|\rho_k^{\frac12}w_1\|^2_{L^2}+\nu^{\frac14}|k|^{\frac12}|\la+ k|^{\frac34}|c^{(2)}_1(\la)|^2\|\rho^{-\frac14}_k w_1\|^2_{L^2}\leq C\nu^{-\frac43}|k|^{\frac56}\|F_2\|^2_{L^2}.
\end{align}
Furthermore, we use H\"older's inequality to get
\begin{equation}\label{om21}
\begin{aligned}
|\om^{(2)}_1(t, k,y)|^2\leq& \Big|\int_{\R}|c^{(2)}_1(\la)w_1(\la, k,y)| d\la\Big|^2\\
\leq& \int_{\R}|\la+ k|^{\frac32}|c^{(2)}_1(\la)|^2\rho_k(y)|w_1(\la, k,y)|^2\\
&+\nu^{\frac14}| k|^{\frac12}|\la+ k|^{\frac34}|c^{(2)}_1(\la)|^2\rho^{-\frac12}_k (y)|w_1(\la, k,y)|^2 d\la\\
&\times \int_{\R} \big(|\la+ k|^{\frac32}\rho_k(y)+(\nu k^2)^{\frac14}|\la+ k|^{\frac34}\rho^{-\frac12}_k(y)\big)^{-1} d\la.
\end{aligned}
\end{equation}
It is known from \cite[P159]{Chen-Li-Wei-Zhang} that for $10\nu\leq |k|<1$, there still holds
\begin{align*}
\int_{\R}(|\la+ k|^{\frac32}\rho_k(y)+(\nu k^2)^{\frac14}|\la+ k|^{\frac34}\rho^{-\frac12}_k(y))^{-1} d\la
=C\nu^{-\frac16}| k|^{-\frac13}.
\end{align*}
In terms of \eqref{om21} and \eqref{c21rhow1+c21rho-14w1L2}, we have
\begin{align*}
\|\om^{(2)}_1(t)\|^2_{L^2_y}
\leq&C\nu^{-\frac16}| k|^{-\frac13}\int_{\R}\int^1_{-1}|\la+ k|^{\frac32}|c^{(2)}_1(\la)|^2\rho_k(y)|w_1(\la, k,y)|^2\\
&+\nu^{\frac14}| k|^{\frac12}|\la+ k|^{\frac34}|c^{(2)}_1(\la)|^2\rho^{-\frac12}_k (y)|w_1(\la, k,y)|^2dy d\la \\
=& C\nu^{-\frac16}| k|^{-\frac13}\int_{\R}|\la+ k|^{\frac32}|c^{(2)}_1(\la)|^2\|\rho^{\frac12}_k w_1\|^2_{L^2}
+\nu^{\frac14}| k|^{\frac12}|\la+ k|^{\frac34}|c^{(2)}_1(\la)|^2\|\rho^{-\frac14}_k w_1\|^2_{L^2}d\la\\
\leq& C\nu^{-\frac16}| k|^{-\frac13}\int_{\R}\nu^{-\frac43}|k|^{\frac56}\|F_2(\la)\|^2_{L^2} d\la\sim \nu^{-\frac32}| k|^{\frac12}\|f_2\|^2_{L^2 L^2}.
\end{align*}
This inequality  gives
\begin{align}\label{om21L2f2L2}
\nu^{\frac12} | k|^{-\frac12}\|\om^{(2)}_1\|^2_{L^\infty L^2}\leq C\nu^{\frac12} | k|^{-\frac12}\nu^{-\frac32}| k|^{\frac12}\|f_2\|^2_{L^2L^2} =C\nu^{-1}\|f_2\|^2_{L^2L^2}.
\end{align}
Similarly, we have
\begin{align}\label{om22L2f2L2}
\nu^{\frac12} | k|^{-\frac12}\|\om^{(2)}_2\|^2_{L^\infty L^2}\leq C\nu^{-1}\|f_2\|^2_{L^2L^2}.
\end{align}

Combining \eqref{esti: omNaL2f1f2}--\eqref{omom12L2f1L2}, \eqref{om21L2f2L2} and \eqref{om22L2f2L2}, we finally obtain
\begin{align*}
\nu^{\frac12} | k|^{-\frac12}\|\om\|^2_{L^\infty L^2}\leq& \|\om_{Na}\|^2_{L^\infty L^2}+\sum_{1\leq j,l\leq 2}\nu^{\frac12} | k|^{-\frac12}\|\om^{(j)}_{l}\|^2_{L^\infty L^2}\lesssim \nu^{-\frac13}| k|^{\frac43}\|f_1\|^2_{L^2 L^2}+\nu^{-1}\|f_2\|^2_{L^2 L^2}.
\end{align*}
This completes the proof of Proposition \ref{Prop: inhomo}.
\end{proof}

Next, we establish the space-time estimates for the homogeneous problem.
\begin{proposition}\label{Prop homo pro}
Let $10\nu\leq |k|< 1$ and $\om_{H}$ be a solution of \eqref{equ: omH} with $\om^{in}\in H^1(I)$ and $\lan \om^{in}, e^{\pm  k y}\ran=0$. Then there exists a constant $C>0$ independent of $\nu$, $ k$ such that
\begin{equation}\label{homo om}
\begin{aligned}
&\nu^{\frac13} | k|^{\frac23}\|\rho^{\frac12}_k \om_{H}\|^2_{L^2L^2}+|k|\|u_{H}\|^2_{L^2 L^2}+\nu^{\frac12} | k|^{\frac12}\|\om_{H}\|^2_{L^2L^2}+\nu^{\frac12} | k|^{-\frac12}\|\om_{H}\|^2_{L^\infty L^2}\\
\leq& C\|\om^{in}\|^2_{L^2}+C\nu^{\frac23}| k|^{-\frac23}\|\pa_y\om^{in}\|^2_{L^2}.
\end{aligned}
\end{equation}
\end{proposition}
\begin{proof}
It is easy to see that
\begin{align*}
&\om^{(1)}_{H}(t, k,y)=e^{-(\nu k^2)^{1/3}t}e^{-it k y}\om^{in}(k,y), \quad t>0
\end{align*}
satisfies
\begin{align}\label{equ: omH1}
&(\pa_t-\nu(\pa^2_y- k^2)+i k y)\om^{(1)}_{H}=(\nu k^2-(\nu  k^2)^{\frac13})\om^{(1)}_{H}-\nu \pa^2_y\om^{(1)}_{H},\quad \om^{(1)}_H(0)=\om^{in}(k,y).
\end{align}
Let $\om=\om_{H}-\om^{(1)}_{H}$. The following system arises from \eqref{equ: omH} and \eqref{equ: omH1}
\begin{equation*}
\left\{
\begin{aligned}
&\pa_t\om-\nu(\pa^2_y-k^2)\om+iky\om=-(\nu k^2-(\nu  k^2)^{\frac13})\om^{(1)}_{H}+\nu \pa^2_y\om^{(1)}_{H},\\
&\om\big|_{t=0}=0,\quad \lan \om, e^{\pm ky}\ran=-\lan \om^{(1)}_{H},e^{\pm ky}\ran.
\end{aligned}
\right.
\end{equation*}
We decompose $\om$ into
\begin{align*}
\om=\om^{(2)}_{H}+\om^{(3)}_{H},
\end{align*}
where $\om^{(2)}_H$ and $\om^{(3)}_H$ solve
\begin{equation}\left\{\begin{aligned}
&(\pa_t-\nu(\pa^2_y- k^2)+i k y)\om^{(2)}_{H}=-(\nu k^2-(\nu  k^2)^{\frac13})\om^{(1)}_{H}+\nu \pa^2_y\om^{(1)}_{H},\\
&\om^{(2)}_{H}|_{t=0}=0,\qquad \lan \om^{(2)}_H, e^{\pm  k y}\ran =0,\end{aligned}\right.\label{equ: omH2}\end{equation}
and
\begin{equation}\left\{
\begin{aligned}
&(\pa_t-\nu(\pa^2_y- k^2)+i k y)\om^{(3)}_H=0, \\
&\om^{(3)}_H|_{t=0}=0, \qquad \lan\om^{(3)}_{H}, e^{\pm  k y}\ran =-\lan \om^{(1)}_{H},e^{\pm ky}\ran,\label{equ: omH3}
\end{aligned}
\right.
\end{equation}
respectively. Denote
\begin{align*}
u^{(j)}_{H}=(\pa_y, -i k)\psi^{j}_H, \quad \psi^{(j)}_{H}=(\pa^2_y- k^2)^{-1}\om^{(j)}_H, \quad j=0,1,2,3.
\end{align*}
We divide the proof into two steps as follows.

\textbf{Step 1.} Estimates of $\om^{(j)}_{H}$, $j=1,2$.
By Proposition \ref{Prop: inhomo}, we have for $10\nu\leq|k|<1$,
\begin{equation}\label{esti: om2HL2L2}
\begin{aligned}
&\nu^{\frac13} | k|^{\frac23}\|\rho^{\frac12}_k \om^{(2)}_{H}\|^2_{L^2 L^2}+| k|\|u^{(2)}_{H}\|^2_{L^2 L^2}+\nu^{\frac12} | k|^{\frac12} \|\om^{(2)}_{H}\|^2_{L^2L^2}+\nu^{\frac12} | k|^{-\frac12}\|\om^{(2)}_{H}\|^2_{L^\infty L^2}\\
\leq & C\big(\nu^{-\frac13}| k|^{-\frac23}\|(\nu  k^2-(\nu  k^2)^{\frac13})\om^{(1)}_H\|^2_{L^2 L^2}+\nu \|\pa_y \om^{(1)}_H\|^2_{L^2L^2}\big)\\
\leq & C((\nu  k^2)^{\frac13} \|\om^{(1)}_H\|^2_{L^2 L^2}+\nu \|\pa_y \om^{(1)}_H\|^2_{L^2L^2}).
\end{aligned}
\end{equation}
Hence, it suffices to estimate $\om^{(1)}_{H}$. It follows from \cite[P161-162]{Chen-Li-Wei-Zhang} that
\begin{align}
&\|\om^{(1)}_H(t)\|_{L^2}=e^{-(\nu  k^2)^{1/3}t}\|\om^{in}\|_{L^2},\label{esti: om1HL2}\\
&(\nu  k^2)^{\frac13}\|\om^{(1)}_{H}\|^2_{L^2 L^2}=\|\om^{in}\|^2_{L^2}/2,\label{esti: on1HL2L2}\\
&\nu \|\pa_y \om^{(1)}_{H}\|^2_{L^2L^2}\leq \nu^{\frac23}| k|^{-\frac23}\|\pa_y \om^{in}\|^2_{L^2}+2\|\om^{in}\|^2_{L^2}\label{payom1HL2L2}
\end{align}
and
\begin{equation}
\begin{aligned}\label{u1HL2L2}
\int_{\R} \|u^{(1)}_{H}(t)\|^2_{L^2} dt\leq \sum_{j=1}^{+\infty}\frac{2\pi}{| k|}\frac{\|\om^{in}\|^2_{L^2}}{(\pi j/2)^2+| k|^2}\leq\frac{2\pi}{| k|}\|\om^{in}\|^2_{L^2}.
\end{aligned}
\end{equation}

Inserting \eqref{esti: on1HL2L2} and \eqref{payom1HL2L2} into \eqref{esti: om2HL2L2}, we obtain
\begin{equation}\label{esti:omH2}
\begin{aligned}
&\nu^{\frac13} | k|^{\frac23}\|\rho^{\frac12}_k \om^{(2)}_{H}\|^2_{L^2 L^2}+| k|\|u^{(2)}_{H}\|^2_{L^2 L^2}+\nu^{\frac12} | k|^{\frac12} \|\om^{(2)}_{H}\|^2_{L^2L^2}+\nu^{\frac12} | k|^{-\frac12}\|\om^{(2)}_{H}\|^2_{L^\infty L^2}\\
\lesssim & \|\om^{in}\|^2_{L^2}+\nu^{\frac23}| k|^{-\frac23}\|\pa_y \om^{in}\|^2_{L^2}.
\end{aligned}
\end{equation}
By \eqref{esti: om1HL2}, \eqref{esti: on1HL2L2} and \eqref{u1HL2L2}, we have
\begin{equation}\label{esti:omH1}
\begin{aligned}
&(\nu  k^2)^{\frac13}\|\rho^{\frac12}_k \om^{(1)}_H\|^2_{L^2L^2}+| k|\|u^{(1)}_H\|^2_{L^2L^2}+\nu^{\frac12}|k|^{\frac12}\|\om^{(1)}_H\|^2_{L^2 L^2}+\nu^{\frac12}|k|^{-\frac12} \|\om^{(1)}_H\|^2_{L^\infty L^2}\\
\lesssim & (\nu  k^2)^{\frac13}\|\om^{(1)}_H\|^2_{L^2 L^2}+\|\om^{(1)}_{H}\|^2_{L^\infty L^2}\leq C\|\om^{in}\|^2_{L^2}.
\end{aligned}
\end{equation}

\textbf{Step 2.} Estimate of $\om^{(3)}_{H}$.
Let
\begin{align}\label{expre: wcj}
w^{(3)}_{H}(\la, k,y):=\int^{+\infty}_0\om^{(3)}_H(t, k,y) e^{-it\la}dt. 
\end{align}
It is easy to verify that
\begin{align*}
&(i\la-\nu (\pa^2_y- k^2)+i k y) w^{(3)}_{H}(\la, k,y)=0.
\end{align*}
Define
\begin{align}\label{expre: c1la}
\widetilde{c}_l(\la)=-\int^1_{-1} \frac{\sinh  k(1+(-1)^{l+1}y)}{\sinh 2 k} w^{(3)}_{H}(\la, k,y) dy,\quad l=1,2.
\end{align}
Then, we have
\begin{align*}
w^{(3)}_{H}=-\widetilde{c}_1(\la)w^{(3)}_{H,1}-\widetilde{c}_2(\la)w^{(3)}_{H,2},
\end{align*}
where $w^{(3)}_{H,1}=(\pa^2_y-k^2)\phi_1$, $w^{(3)}_{H,2}=(\pa^2_y-k^2)\phi_2$, and $\phi_1$, $\phi_2$ are the solutions of homogeneous OS equations \eqref{equ: psi1} and \eqref{equ: psi2} with $\epsilon=0$.

By Proposition \ref{Lemma: Nonslip wLinfty wL1},
we infer
\begin{equation}\label{wH3L2}
\begin{aligned}
\|w^{(3)}_{H}(\la,k,y)\|_{L^2_y}\leq &|\widetilde{c}_1(\la)|\|w^{(3)}_{H,1}\|_{L^2_y}+|\widetilde{c}_2(\la)|\|w^{(3)}_{H,2}\|_{L^2_y}\\
\leq & C\nu^{-\frac14}| k|^{\frac14}\big(|\widetilde{c}_1(\la)|(1+|\la/ k+1|)^{\frac14}+|\widetilde{c}_2(\la)|(1+|-\la/ k+1|)^{\frac14}\big).
\end{aligned}
\end{equation}
It follows from \eqref{esti: psiL2omL2} and \eqref{esti: w1+w2L1} that
\begin{equation}\label{uH3L2}
\begin{aligned}
\|( k,\pa_y)(\pa_y^2- k^2)^{-1} w^{(3)}_{H}(\la,k,y)\|_{L^2_y}
\leq &|\widetilde{c}_1(\la)|\|w^{(3)}_{H,1}\|_{L^1_y}+|\widetilde{c}_2(\la)|\|w^{(3)}_{H,2}\|_{L^1_y}\\
\lesssim &(|\widetilde{c}_1(\la)|+|\widetilde{c}_2(\la)|).
\end{aligned}
\end{equation}
By \eqref{rho12wL2} and $L=(k/\nu)^{\frac13}$, we have
\begin{equation}\label{weight rhowL2}
\begin{aligned}
\|\rho^{\frac12}_k w^{(3)}_{H}(\la,k,y)\|_{L^2_y}\leq& |\widetilde{c}_1(\la)|\|\rho^{\frac12}_k w^{(3)}_{H,1}\|_{L^2}+|\widetilde{c}_2(\la)|\|\rho^{\frac12}_k w^{(3)}_{H,2}\|_{L^2}\\
\leq & C\nu^{-\frac16}|k|^{\frac16}(|\widetilde{c}_1(\la)|+|\widetilde{c}_2(\la)|).
\end{aligned}
\end{equation}
Using Plancherel's theorem and summing up \eqref{wH3L2}--\eqref{weight rhowL2}, we arrive at
\begin{align*}
&(\nu  k^2)^{\frac13}\|\rho^{\frac12}_k \om^{(3)}_{H}\|^2_{L^2L^2}+|k|\|u^{(3)}_{H}\|^2_{L^2L^2}+\nu^{\frac12}  |k|^{\frac12}\|\om^{(3)}_{H}\|^2_{L^2L^2}\\
\sim& (\nu  k^2)^{\frac13}\big\|\|\rho^{\frac12}_k w^{(3)}_{H}\|^2_{L^2_y}\big\|^2_{L^2_{\la}}+|k|\big\|\|( k,\pa_y)(\pa_y^2- k^2)^{-1}w^{(3)}_{H}\|^2_{L^2_y} \big\|^2_{L^2_{\la}} +\nu^{\frac12}|k|^{\frac12}\big\|\|w^{(3)}_{H}\|^2_{L^2_y}\big\|^2_{L^2_{\la}}\\
\leq& C|k|\big(\|(1+|\la/ k+1|)^{\frac14} \widetilde{c}_1\|^2_{L^2_{\la}}+\|(1+|\la/ k-1|)^{\frac14}\widetilde{c}_2\|^2_{L^2_{\la}}\big).
\end{align*}
Thanks to $\om^{(3)}_H(t)=\frac{1}{2\pi}\int_{\R} w^{(3)}_{H}(\la)e^{it\la}d\la$ and \eqref{wH3L2}, we get
\begin{align*}
\nu^{\frac12}| k|^{-\frac12}\|\om^{(3)}_{H}\|^2_{L^\infty L^2}\leq &\nu^{\frac12}| k|^{-\frac12}\big\|\|w^{(3)}_{H}(\la)\|_{L^2_y}\big\|^2_{L^1_{\la}}\\
\lesssim &\nu^{\frac12} k^{-\frac12}\nu^{-\frac12}| k|^{\frac12}\big(\|(1+|\la/ k+1|)^{\frac14} \widetilde{c}_1\|^2_{L^1_{\la}}+\|(1+|\la/ k-1|)^{\frac14}\widetilde{c}_2\|^2_{L^1_{\la}}\big)\\
=&\big(\|(1+|\la/ k+1|)^{\frac14} \widetilde{c}_1\|^2_{L^1_{\la}}+\|(1+|\la/ k-1|)^{\frac14}\widetilde{c}_2\|^2_{L^1_{\la}}\big)\\
\lesssim & |k|\big(\|(1+|\la/ k+1|) \widetilde{c}_1\|^2_{L^2_{\la}}+\|(1+|\la/ k-1|)\widetilde{c}_2\|^2_{L^2_{\la}}\big).
\end{align*}
By \eqref{expre: c1la} and the boundary condition
\begin{align*}
\lan\om^{(3)}_{H}, e^{\pm  k y}\ran =-\lan \om^{(1)}_{H},e^{\pm ky}\ran,
\end{align*}
we have
\begin{align}
\widetilde{c}_l(\la)=\int^1_{-1} \frac{\sinh  k(1+(-1)^{l+1}y)}{\sinh 2 k} w^{(1)}_{H}(\la, k,y) dy,\quad l=1,2.
\end{align}
Choosing $g=\om^{in}$ in Lemma \ref{claim}, we arrive at
\begin{align*}
&(\nu  k^2)^{\frac13}\|\rho^{\frac12}_k \om^{(3)}_H\|^2_{L^2L^2}+ |k|\|u^{(3)}_{H}\|^2_{L^2L^2}+\nu^{\frac12}|k|^{\frac12} \|\om^{(3)}_{H}\|^2_{L^2L^2} +\nu^{\frac12}|k|^{-\frac12}\|\om^{(3)}_{H}\|^2_{L^\infty L^2}\\
\leq& C|k|(\|(1+|\la/ k+1|)\widetilde{c}_1\|^2_{L^2_{\la}}+\|(1+|\la/ k-1|)\widetilde{c}_2\|^2_{L^2_{\la}})\leq C\|\om^{in}\|^2_{L^2}.
\end{align*}
This together with \eqref{esti:omH2} and \eqref{esti:omH1} completes the proof of \eqref{homo om}.
\end{proof}
\section{Bilinear Estimates and Nonlinear Transition Threshold}\label{nonlinear tran thre}
In this section, we shall apply the space-time estimates  and  a delicate analysis on the interactions between the low, intermediate and high frequencies of $u$ and $\om$ to prove Theorem \ref{Th: tran thre}, which is the transition threshold of the Couette flow on $\R\times [-1,1]$ with non-slip boundary condition.

Recall
\begin{equation*}
E_k=\left\{
\begin{aligned}
&\|\om_k\|_{L^\infty_t L^2_y}+\nu^{\frac12}\|\om_k\|_{L^2_t L^2_y}+\|u_k\|_{L^\infty_t L^\infty_y}+\nu^{\frac12}\|u_k\|_{L^2_tL^2_y}, \quad 0<|k|< 10\nu,\\
&\|(1-|y|)^{\frac12}\om\|_{L^\infty_t L^2_y}+\nu^{\frac14}|k|^{-\frac14}\|\om_k\|_{L^\infty_t L^2_y}+\nu^{\frac14}|k|^{\frac14}\|\om_k\|_{L^2_t L^2_y}+\|u_k\|_{L^\infty_t L^\infty_y}\\
&+|k|^{\frac12}\|u_k\|_{L^2_tL^2_y}, \qquad\qquad\qquad\qquad\qquad\qquad\qquad\qquad\qquad \qquad   10\nu\leq|k|<1,\\
&\|(1-|y|)^{\frac12}\om_k\|_{L^\infty_t L^2_y}+\nu^{\frac14}|k|^{\frac12}\|\om_k\|_{L^2_t L^2_y}+|k|^{\frac12}\|u_k\|_{L^\infty_t L^\infty_y}+|k|\|u_k\|_{L^2_t L^2_y},\,\, |k|\geq 1.
\end{aligned}
\right.
\end{equation*}
\begin{proof}[Proof of Theorem \ref{Th: tran thre}]
It follows from \eqref{Th nuxi2leq1 k>1} and Theorem \ref{Th: linear problem} that
\begin{equation}\label{esti: Ek}
\begin{aligned}
\|E_{k}\|_{L^1_k}\leq \bbn{\min\lr{\nu^{-\frac16}|k|^{\frac23},\nu^{-\frac{1}{2}}}f_1(k)}_{L_{k}^{1}L^2_tL^2_y}+\nu^{-\frac12} \|f_2(k)\|_{L^1_kL^2_tL^2_y} +\|u^{in}_k\|_{L^1_kH^2_y}.
\end{aligned}
\end{equation}
We claim
\begin{align}
&\bbn{\min\lr{\nu^{-\frac16}|k|^{\frac23},\nu^{-\frac{1}{2}}}f_1(k)}_{L_{k}^{1}L^2_tL^2_y}+\nu^{-\frac12}\|f_2(k)\|_{L^1_kL^2_tL^2_y} \leq C\nu^{-\frac12}\|E_k\|^2_{L^1_k}.\label{f1kLk1L2tL2y}
\end{align}
Inserting \eqref{f1kLk1L2tL2y} into \eqref{esti: Ek}, we arrive at
\begin{align*}
\|E_k\|_{L^1_k}\leq &\nu^{-\frac14}\||k|^{\frac12}f_1\|_{L^1_kL^2_tL^2_y}+\nu^{-\frac12}\|f_2\|_{L^1_kL^2_tL^2_y}+\|u^{in}_k\|_{L^1_k H^2_y}\\
\leq& C\nu^{-\frac12}\|E_k\|^2_{L^1_k}+\|u^{in}_k\|_{L^1_k H^2_y}.
\end{align*}
Due to the condition $\|u^{in}\|_{H^2}\leq c\nu^{\frac12}$, we deduce
\begin{align*}
\|u^{in}_k\|_{L^1_k H^2_y}\leq \int_{\R}(1+|k|^2)\|u^{in}_k\|_{H^2_y}\frac{1}{(1+|k|^2)} dk\leq \|u^{in}\|_{H^2(\R\times I)}\leq c\nu^{\frac12}.
\end{align*}
Choosing $c$ suitably small, we obtain by the continuation argument
$$\|E_k\|_{L^1_k}\leq C\nu^{\frac12},$$
which completes the proof of Theorem \ref{Th: tran thre}.

Next, we prove the claim \eqref{f1kLk1L2tL2y}. From the definition of $E_k$, the estimates $\|l u^1_l\|_{L^2_tL^2_y}\leq E_l$ and $\|(1-|y|)^{\frac12}\om_{k-l}\|_{L^\infty_t L^2_y}\leq E_{k-l}$ hold true for $l\in\R$ and $k\in\R$. Then, in terms of Hardy's inequality and the condition $\mathrm{div} u=0$, we deduce that
\begin{equation}\label{f2kL1kL2tL2y}
\begin{aligned}
\nu^{-\frac12}\|f_2\|_{L^1_kL^2_tL^2_y}\leq &\nu^{-\frac12}\iint_{l,k}\Big\|\frac{u^2_l}{(1-y)^{\frac12}}\Big\|_{L^2_t L^\infty_y}\|(1-|y|)^{\frac12}\om_{k-l}\|_{L^\infty_t L^2_y}dldk\\
\leq&\nu^{-\frac12} \iint_{l,k}\|\pa_y u^2_l\|_{L^2_tL^2_y}\|(1-|y|)^{\frac12}\om_{k-l}\|_{L^\infty_t L^2_y}dldk\\
=&\nu^{-\frac12}\iint_{l,k}\|l u^1_l\|_{L^2_tL^2_y}\|(1-|y|)^{\frac12}\om_{k-l}\|_{L^\infty_t L^2_y}dldk\\
\leq&\nu^{-\frac12}\|E_k\|_{L^1_k}\|E_l\|_{L^1_l}.
\end{aligned}
\end{equation}

For the estimate of $f^1_k$, we first analyze the two components $u_l$ and $w_{k-l}$ in $f^1_k$.
Decompose $\R$ into $I_1+I_2+I_3$, where
\begin{align*}
I_{1}=&\lr{x:0\leq |x|\leq 10 \nu},\\
I_{2}=&\lr{x:10\nu<|x|\leq 1},\\
I_{3}=&\lr{x:|x|> 1}.
\end{align*}
From the definition of $E_k$, it is clear that
\begin{align}
&\n{u_{l}^{1}}_{L_{t}^{\infty}L_{y}^{\infty}}\lesssim
\min\lr{1,|l|^{-\frac{1}{2}}}E_{l},
\quad&
\n{\om_{k}}_{L_{t}^{2}L_{y}^{2}}\lesssim
\begin{cases}
\nu^{-\frac{1}{2}}E_{k},\ &k\in I_{1},\\
\nu^{-\frac{1}{4}}|k|^{-\frac{1}{4}}E_{k},\ &k\in I_{2},\\
\nu^{-\frac{1}{4}}|k|^{-\frac{1}{2}}E_{k},\ &k\in I_{3}.\\
\end{cases}\label{equ:u00,w22}
\end{align}
By \eqref{esti: psiLinfty inte}, we have $\|u_{l}^1\|_{L^2_t L^\infty_y}\leq \|u_{l}\|^{\frac12}_{L^2_tL^2_y}\|\om_l\|^{\frac12}_{L^2_t L^2_y}$, which implies
\begin{align}
&\n{u_{l}^{1}}_{L_{t}^{2}L_{y}^{\infty}}\lesssim
\begin{cases}
\nu^{-\frac{1}{2}}E_{l},\ &l\in I_{1},\\
\nu^{-\frac{1}{8}}|l|^{-\frac{3}{8}}E_{l},\ & l\in I_{2},\\
\nu^{-\frac{1}{8}}|l|^{-\frac{3}{4}}E_{l},\ & l\in I_{3},
\end{cases}
\quad&  \n{\om_{k}}_{L_{t}^{\infty}L_{y}^{2}}\lesssim
\begin{cases}
E_{k} &k\in I_{1},\\
\nu^{-\frac{1}{4}}|k|^{\frac{1}{4}}E_{k},\ &k\in I_{2}.
\end{cases}\label{equ:u20,w02}
\end{align}

Since $u$ and $\om$ exhibit distinct space-time estimates across different frequencies, we divide the estimate of $f_{k}^1$ into the following nine parts:
\begin{equation}\label{f1 esti}
\begin{aligned}
\bbn{\min\lr{\nu^{-\frac16}|k|^{\frac23},\nu^{-\frac{1}{2}}}f_{k}^{1}}_{L_{k}^{1}L_{t}^{2}L_{y}^{2}}
=& \iint_{\R^{2}}  \min\lr{\nu^{-\frac16}|k|^{\frac23},\nu^{-\frac{1}{2}}}\n{u_{l}^{1}\om_{k-l}}_{L_{t}^{2}L_{y}^{2}}dl dk\\
=&\iint_{\R^{2}}\min\lr{\nu^{-\frac16}|k+l|^{\frac23},\nu^{-\frac{1}{2}}}\n{u_{l}^{1}\om_{k}}_{L_{t}^{2}L_{y}^{2}}dl dk\\
=&\sum_{i,j=1}^{3}\iint_{I_{i,j}}\min\lr{\nu^{-\frac16}|k+l|^{\frac23},\nu^{-\frac{1}{2}}}\n{u_{l}^{1}\om_{k}}_{L_{t}^{2}L_{y}^{2}}dl dk,
\end{aligned}
\end{equation}
where $I_{i,j}=\lr{(l,k):l\in I_{i},k\in I_{j}}$. For clarity, we divide the proof into three cases, as shown in the following table:

\begin{center}
\begin{tabular}{|c|c|c|c|}
	\hline Case \textrm{I}: & $ I_{1,1}= I_1\times I_1$ &  $ I_{2,1}= I_2\times I_1$  &  $ I_{3,1}= I_3\times I_1$  \\
   \hline Case \textrm{II}: & $ I_{1,2}= I_1\times I_2$ & $ I_{2,2}= I_2\times I_2$ &$ I_{3,2}= I_3\times I_2$\\
    \hline Case \textrm{III}: & $ I_{1,3}= I_1\times I_3$ & $ I_{2,3}= I_2\times I_3$ &$ I_{3,3}= I_3\times I_3$\\
   \hline
\end{tabular}
\end{center}

\textbf{Case \textrm{I}: $I_{1,1}\cup I_{2,1}\cup I_{3,1}$.} Thanks to H\"older's inequality, we have
\begin{align*}
&\sum_{i=1}^{3}\iint_{I_{i,1}}\min\lr{\nu^{-\frac16}|k+l|^{\frac23},\nu^{-\frac{1}{2}}}\n{u_{l}^{1}\om_{k}}_{L_{t}^{2}L_{y}^{2}}dl dk\\
\leq&\nu^{-\frac{1}{6}}\sum_{i=1}^{3}\iint_{I_{i,1}}|k+l|^{\frac{2}{3}}\n{u_{l}^{1}}_{L_{t}^{2}L_{y}^{\infty}}\n{\om_{k}}_{L_{t}^{\infty}L_{y}^{2}}dl dk
=:A_{1,1}+A_{2,1}+A_{3,1}.
\end{align*}

We use \eqref{equ:u20,w02} to obtain
\begin{align*}
A_{1,1}+A_{2,1}+A_{3,1}\lesssim &\nu^{-\frac{1}{6}}\iint_{I_{1,1}}\frac{|k+l|^{\frac{2}{3}}}{\nu^{\frac12}} E_{l}E_{k} dl dk+\nu^{-\frac{1}{6}}\iint_{I_{2,1}}\frac{ |k+l|^{\frac23}}{\nu^{\frac{1}{8}}|l|^{\frac{3}{8}}}E_{l}E_{k}dldk\\
&+ \nu^{-\frac{1}{6}}\iint_{I_{3,1}}\frac{|k+l|^{\frac23}}{\nu^{\frac18}|l|^{\frac34}}E_{l}E_{k}dldk\\
\lesssim &\nu^{-\frac{7}{24}}\n{E_{l}}_{L_{l}^{1}}\n{E_{k}}_{L_{k}^{1}}\lesssim \nu^{-\frac{1}{2}}\n{E_{l}}_{L_{l}^{1}}\n{E_{k}}_{L_{k}^{1}},
\end{align*}
where we have used that
\begin{align*}
&|k+l|^{\frac{2}{3}}\leq \nu^{\frac23}, \quad (l,k)\in I_{1,1},\\
&\frac{|k+l|^{\frac{2}{3}}}{|l|^{\frac{3}{8}}}\lesssim \frac{\nu^{\frac{2}{3}}+|l|^{\frac{2}{3}}}{|l|^{\frac{3}{8}}}
\lesssim \frac{\nu^{\frac{2}{3}}}{\nu^{\frac{3}{8}}}+|l|^{\frac{7}{24}}\lesssim 1,\quad (l,k)\in I_{2,1},
\end{align*}
and
\begin{align*}
&\frac{|k+l|^{\frac23}}{|l|^{\frac{3}{4}}}\lesssim \frac{|k|^{\frac{2}{3}}+|l|^{\frac23}}{|l|^{\frac{3}{4}}}\lesssim \nu^{\frac23}+ \frac{|l|^{\frac23}}{|l|^{\frac{3}{4}}}\lesssim 1,\quad (l,k)\in I_{3,1}.
\end{align*}

\textbf{Case \textrm{II}: $I_{1,2}\cup I_{2,2}\cup I_{3,2}$.} It is easy to see that 
\begin{align*}
&\sum_{i=1}^{3}\iint_{I_{i,2}}\min\lr{\nu^{-\frac16}|k+l|^{\frac23},\nu^{-\frac{1}{2}}}\n{u_{l}^{1}\om_{k}}_{L_{t}^{2}L_{y}^{2}}dl dk\\
\leq&\sum_{i=1}^{3}\nu^{-\frac{1}{6}}\iint_{I_{i,2}}|k+l|^{\frac{2}{3}}\n{u_{l}^{1}\om_{k}}_{L_{t}^{2}L_{y}^{2}}dl dk=:A_{1,2}+A_{2,2}+A_{3,2}.
\end{align*}

For $A_{1,2}$, we use \eqref{equ:u00,w22} for $(l,k)\in I_{1,2}$ to obtain
\begin{align*}
A_{1,2}\leq &\nu^{-\frac{1}{6}}\iint_{I_{1,2}}|k+l|^{\frac{2}{3}}\n{u_{l}^{1}}_{L_{t}^{\infty}L_{y}^{\infty}}
\n{\om_{k}}_{L_{t}^{2}L_{y}^{2}}dl dk\\
\lesssim &\nu^{-\frac{1}{6}}\iint_{I_{1,2}}\frac{|k+l|^{\frac{2}{3}}}{\nu^{\frac{1}{4}}|k|^{\frac{1}{4}}}
E_{l}E_{k}dldk
\lesssim \nu^{-\frac{1}{2}}\|E_{l}\|_{L^1_l}\|E_{k}\|_{L^1_k},
\end{align*}
where we have used that
\begin{align*}
\frac{|k+l|^{\frac23}}{|k|^{\frac{1}{4}}}\lesssim \frac{|k|^{\frac23}+|l|^{\frac23}}{|k|^{\frac14}}\lesssim \frac{|k|^{\frac23}+|k|^{\frac23}}{|k|^{\frac14}}\lesssim |k|^{\frac{5}{12}}\lesssim 1, \quad (l,k)\in I_{1,2}.
\end{align*}

It is clear that $A_{2,2}$ can be bounded by
\begin{align*}
A_{2,2}\leq&\nu^{-\frac{1}{6}}\iint_{I_{2,2}}(|k|^{\frac{2}{3}}+|l|^{\frac{2}{3}})\n{u_{l}^{1}\om_{k}}_{L_{t}^{2}L_{y}^{2}}dl dk.
\end{align*}
For the term $|k|^{\frac{2}{3}}\n{u_{l}^{1}\om_{k}}_{L_{t}^{2}L_{y}^{2}}$ with $(l,k)\in I_{2,2}$, we have by \eqref{equ:u00,w22} that
\begin{align*}
\nu^{-\frac16}|k|^{\frac{2}{3}}\n{u_{l}^{1}\om_{k}}_{L_{t}^{2}L_{y}^{2}}\leq \nu^{-\frac16}\frac{|k|^{\frac23}}{\nu^{\frac14}|k|^{\frac14}}E_{l}E_{k}=\nu^{-\frac{5}{12}}|k|^{\frac{5}{12}}E_{l}E_{k}\leq \nu^{-\frac{1}{2}}E_{l}E_{k}.
\end{align*}
For the term $|l|^{\frac{2}{3}}\n{u_{l}^{1}\om_{k}}_{L_{t}^{2}L_{y}^{2}}$, we use both \eqref{equ:u00,w22} and \eqref{equ:u20,w02} to obtain
\begin{align*}
\nu^{-\frac{1}{6}}|l|^{\frac{2}{3}}\n{u_{l}^{1}\om_{k}}_{L_{t}^{2}L_{y}^{2}}\leq &\nu^{-\frac{1}{6}}|l|^{\frac{2}{3}}(\n{u_{l}^{1}}_{L^\infty_t L^\infty_y }\n{\om_k}_{L^2_tL^2_y})^{\frac12} (\n{u_{l}^{1}}_{L^2_t L^\infty_y }\n{\om_k}_{L^\infty_tL^2_y})^{\frac12} \\ \leq&\nu^{-\frac{1}{6}}|l|^{\frac{2}{3}}(\nu^{-\frac14}|k|^{-\frac14}E_{l}E_{k})^{\frac12} (\nu^{-\frac38}|l|^{-\frac38}|k|^{\frac14}E_lE_k)^{\frac12}\\
=& \nu^{-\frac{23}{48}}|l|^{\frac{23}{48}}E_{l}E_{k}\leq \nu^{-\frac{1}{2}}E_{l}E_{k}.
\end{align*}
The above two estimates imply that
\begin{align*}
A_{2,2}\lesssim \nu^{-\frac12}\n{E_l}_{L_{l}^{1}}\n{E_k}_{L_{k}^{1}}.
\end{align*}

For $A_{3,2}$, we have
\begin{align*}
A_{3,2}
\leq&\nu^{-\frac{1}{6}}\iint_{I_{3,2}}(|k|^{\frac{2}{3}}+|l|^{\frac{2}{3}})\n{u_{l}^{1}\om_{k}}_{L_{t}^{2}L_{y}^{2}}dl dk.
\end{align*}
Applying \eqref{equ:u00,w22} to $|k|^{\frac{2}{3}}\n{u_{l}^{1}\om_{k}}_{L_{t}^{2}L_{y}^{2}}$, and \eqref{equ:u00,w22}--\eqref{equ:u20,w02} to
$|l|^{\frac{2}{3}}\n{u_{l}^{1}\om_{k}}_{L_{t}^{2}L_{y}^{2}}$, we arrive at
\begin{align*}
A_{3,2}\leq& \nu^{-\frac{1}{6}}\iint_{I_{3,2}}\frac{|k|^{\frac23}}{\nu^{\frac14}|k|^{\frac14}}E_{l}E_{k}dldk\\
&+\nu^{-\frac{1}{6}}\iint_{I_{3,2}}|l|^{\frac23}(\nu^{-\frac14}|k|^{-\frac14}|l|^{-\frac12}E_{l}E_{k})^{\frac13} (\nu^{-\frac38}|l|^{-\frac34}|k|^{\frac14}E_lE_k)^{\frac23} dldk\\
=&\iint_{I_{3,2}}\lrs{\nu^{-\frac{5}{12}}|k|^{\frac{5}{12}}+\nu^{-\frac{1}{2}}|k|^{\frac16}}E_{l}E_{k}dldk
\lesssim \nu^{-\frac12}\n{E_l}_{L_{l}^{1}}\n{E_k}_{L_{k}^{1}},
\end{align*}
where we have used $|k|\leq 1$ for $k\in I_2$.

\textbf{Case \textrm{III}: $I_{1,3}\cup I_{2,3}\cup I_{3,3}$.} Due to $\min\{\nu^{-\frac16}|k+l|^{\frac23},\nu^{-\frac{1}{2}}\}\leq \nu^{-\frac14}|k+l|^{\frac{1}{2}}$, we have
\begin{align*}
&\sum_{i=1}^{3}\iint_{I_{i,3}}\min\lr{\nu^{-\frac16}|k+l|^{\frac23},\nu^{-\frac{1}{2}}}\n{u_{l}^{1}\om_{k}}_{L_{t}^{2}L_{y}^{2}}dl dk\\
\leq&\sum_{i=1}^{3}\iint_{I_{i,3}}\nu^{-\frac14}|k+l|^{\frac{1}{2}}\n{u_{l}^{1}\om_{k}}_{L_{t}^{2}L_{y}^{2}}dl dk
=:A_{1,3}+A_{2,3}+A_{3,3}.
\end{align*}

We use \eqref{equ:u00,w22} to obtain
\begin{align*}
A_{1,3}+A_{2,3}=&\iint_{I_{1,3}\cup I_{2,3}}\nu^{-\frac14}|k+l|^{\frac{1}{2}}\n{u_{l}^{1}\om_{k}}_{L_{t}^{2}L_{y}^{2}}dl dk\\
\leq& \iint_{I_{1,3}\cup I_{2,3}} \frac{|k+l|^{\frac12}}{\nu^{\frac12}|k|^{\frac12}}E_lE_kdl dk\lesssim \nu^{-\frac12}\n{E_l}_{L_{l}^{1}}\n{E_k}_{L_{k}^{1}},
\end{align*}
where we have used
\begin{align*}
&|k+l|^{\frac12}|k|^{-\frac12}\leq 2,\quad (k,l)\in I_{1,3}\cup I_{2,3}.
\end{align*}

For $A_{3,3}$, thanks to \eqref{equ:u00,w22} and $|l||k|\gtrsim |k+l|$ for $(l,k)\in I_{3,3}$, we deduce
\begin{align*}
A_{3,3}=&\iint_{I_{3,3}}\nu^{-\frac14}|k+l|^{\frac{1}{2}}\n{u_{l}^{1}\om_{k}}_{L_{t}^{2}L_{y}^{2}}dl dk\\
\lesssim& \iint_{I_{3,3}} \frac{|k+l|^{\frac{1}{2}}}{\nu^{\frac12}|l|^{\frac12}|k|^{\frac12}}E_l E_k dldk\lesssim \nu^{-\frac12}\n{E_l}_{L_{l}^{1}}\n{E_k}_{L_{k}^{1}}.
\end{align*}
Inserting the estimates of $A_{i,j}$ into \eqref{f1 esti}, we arrive at
\begin{align*}
\bbn{\min\lr{\nu^{-\frac16}|k|^{\frac23},\nu^{-\frac{1}{2}}}f_{k}^{1}}_{L_{k}^{1}L_{t}^{2}L_{y}^{2}}\leq C\nu^{-\frac12}\|E_k\|^2_{L^1_k}.
\end{align*}
This inequality, together with \eqref{f2kL1kL2tL2y}, completes the proof of \eqref{f1kLk1L2tL2y}.
\end{proof}
\section{Enhanced Dissipation Estimates}\label{sec: enhanced dissipation}
In this section, we show that the linearized Navier-Stokes equation
has the enhanced dissipation effect when the frequencies $|k|\geq \nu^{1-}$.

We first point out that the solution of the linear equation
\begin{equation*}\left\{
\begin{aligned}
&\pa_t \om_{Na}+\nu(\pa^2_y-k^2)\om_{Na}+iky \om_{Na}=0,\\
&\om_{Na}(t,k,\pm1)=0, \quad \om_{Na}(0,k,y)=\om^{in}(k,y)
\end{aligned}
\right.
\end{equation*}
has the enhanced dissipation effect on frequencies $|k|\geq \nu^{1-}$.
Actually, together with \cite[Lemma 6.3]{Chen-Li-Wei-Zhang}, Proposition \ref{pro: reso esti FL2} and the Gearhart-Pr\"uss type lemma in \cite{Wei} can imply that the solution $\om_{Na}$ has the following estimate:
\begin{align*}
\|e^{\epsilon(\nu|k|^{2})^{\frac13}t}\om_{Na}\|_{L^\infty_t L^2_y}\leq \|\om^{in}\|_{L^2_y}, \quad  |k|\geq 10\nu,
\end{align*}
where $\epsilon$ is a small constant.

Now, we prove Theorem \ref{Th: linear problem enhance}, which presents the enhanced dissipation estimates with non-slip boundary condition. Recall the linearized equation
\begin{equation}
\left\{
\begin{aligned}\label{equ: omli}
&\pa_t\om-\nu(\pa^2_y-k^2)\om+ik y\om=0, \\
&\om=(\pa^2_y-k^2)\psi,\\
&\psi(t,k,\pm 1)=\psi'(t,k,\pm 1)=0,\\
&\om(0)=\om^{in}(k,y).
\end{aligned}
\right.
\end{equation}
Denote $u=(\pa_y,-ik)(\pa^2_y-k^2)^{-1}\om$. According to the range of $k$: $0\leq |k|<10\nu$, $10\nu\leq |k|<1$ and $|k|\geq 1$, we establish the enhanced dissipation estimates in each range.

For $|k|\geq 1$, it follows \cite[Proposition 6.2-6.3]{Chen-Wei-Zhang} that
\begin{align}\label{k>1}
\|e^{\epsilon(\nu k^2)^{\frac13} t}u\|_{L^\infty_t L^2_y}+\|e^{\epsilon(\nu k^2)^{\frac13} t}u\|_{L^2_t L^2_y}
\leq C|k|^{-1}\|\om^{in}\|^2_{L^2_y}.
\end{align}
For $0\leq |k|< 10\nu$, by the energy estimate method, we easily obtain
\begin{align}\label{k<10nu}
\|e^{\epsilon\nu  t}u\|_{L^\infty_t L^2_y}+\nu^{\frac12}\|e^{\epsilon\nu  t}u\|_{L^2_t L^2_y}
\leq C\|\om^{in}\|_{L^2_y}.
\end{align}
For the frequency range $10\nu\leq |k|<1$, we have
\begin{proposition}\label{Prop homo pro enha}
Let $10\nu\leq |k|<1$ and $\om$ be a solution of \eqref{equ: omli} with  $\om^{in}\in L^2(I)$ and  $\lan \om_0, e^{\pm k y}\ran=0$. Then there exists a constant $C>0$ independent of $\nu$, $k$ such that
\begin{align*}
\|e^{\epsilon(\nu k^2)^{\frac13}t} u\|_{L^\infty_t L^2_y}+|k|^{\frac12}\|e^{\epsilon(\nu k^2)^{\frac13}t} u\|_{L^2_t L^2_y}
\leq C(\nu^{\frac13}|k|^{-\frac13}\|\pa_y \om^{in}\|_{L^2_y}+\|\om^{in}\|_{L^2_y}).
\end{align*}
\end{proposition}
Combining \eqref{k>1}, \eqref{k<10nu} with Proposition \ref{Prop homo pro enha}, we complete the proof of Theorem \ref{Th: linear problem enhance}.
In order to prove Proposition \ref{Prop homo pro enha}, we first give the following lemma.
\begin{lemma}\label{Prop: inhomo enhance}
Let $\om_{I}$ be a solution of \eqref{equ: omI} with $f_1,\,f_2\in L^2$. Then there exists a constant $\epsilon>0$ independent of $\nu$, $k$ such that for $10\nu\leq |k|<1$, we have
\begin{align*}
\|e^{\epsilon\nu k^2 t}u_{I}\|_{L^\infty_t L^2_y}+|k|^{\frac12}\|e^{\epsilon\nu k^2 t}u_{I}\|_{L^2_t L^2_y}
\lesssim \nu^{-\frac16}|k|^{\frac13}\|e^{\epsilon(\nu k^2)^{\frac13}t}f_1\|_{L^2_tL^2_y}+\nu^{-\frac12} \|e^{\epsilon(\nu k^2)^{\frac13}t}f_2\|_{L^2_tL^2_y}.
\end{align*}
\end{lemma}
\begin{proof}
Using energy method, we easily obtain
\begin{align*}
 \frac{1}{2}\frac{d}{dt}\|u_{I}\|^2_{L^2}+\nu \|\om_{I}\|^2_{L^2}
 \leq \nu^{-\frac13}|k|^{\frac23}\|f_1\|^2_{L^2}+\nu^{-1}\|f_2\|^2_{L^2}+C|k|\|\pa_y\psi_{I}\|^2_{L^2}.
\end{align*}
Thanks to
\begin{align*}
\frac{d}{dt}\|e^{\epsilon(\nu k^2)^{\frac13} t}u_{I}\|^2_{L^2}
\leq e^{\epsilon(\nu k^2)^{\frac13} t}\frac{d}{dt}\|u_I\|^2_{L^2}+\epsilon (\nu k^2)^{\frac13}e^{\epsilon(\nu k^2)^{\frac13} t} \|u_{I}\|^2_{L^2},
\end{align*}
we can reduce the estimate of $\|e^{\epsilon(\nu k^2)^{\frac13} t}u_{I}\|_{L^\infty_t L^2_y}$ to $((\nu k^2)^{\frac13} +|k|)\| e^{\epsilon(\nu k^2)^{\frac13} t} u_{I}\|^2_{L^2_tL^2_y}$. For $10\nu\leq |k|<1$, due to
$$(\nu k^2)^{1/3} \| e^{\epsilon(\nu k^2)^{\frac13} t} u_{I}\|^2_{L^2_tL^2_y}\leq |k|\| e^{\epsilon(\nu k^2)^{\frac13} t} u_{I}\|^2_{L^2_tL^2_y},$$
it remains to bound $|k|\| e^{\epsilon(\nu k^2)^{\frac13} t} u_{I}\|^2_{L^2_tL^2_y}$. Therefore, we introduce
\begin{align*}
\widetilde{w}_{I}(\la,k,y)&=\int^{+\infty}_0 \om_{I}(t,k,y) e^{\epsilon(\nu k^2)^{\frac13}t-it\la} dt,\\
\widetilde{F}_{j}(\la,k,y)&=\int^{+\infty}_0 f_{j}(t,k,y)e^{\epsilon(\nu k^2)^{\frac13}t-it\la} dt,\, j=1,2,
\end{align*}
which satisfy
\begin{equation*}
\left\{
\begin{aligned}
&-\nu (\pa^2_y-k^2)\widetilde{w}_{I}+ik \Big(y+\frac{\la}{k}\Big)\widetilde{w}_{I}-\epsilon(\nu k^2)^{\frac13}\widetilde{w}_{I}=-ik \widetilde{F}_1-\pa_y \widetilde{F}_2,\\
&(\pa^2_y-k^2)\widetilde{\phi}_{I}=\widetilde{w}_{I},\\
&\widetilde{\phi}_{I}(\pm1)=\pa_y\widetilde{\phi}_{I}(\pm 1)=0.
\end{aligned}
\right.
\end{equation*}
Let $\widetilde{u}_{I}=(\pa_y\widetilde{\phi}_{I},-ik\widetilde{\phi}_{I})$.
In light of Proposition \ref{lemma:non-slip boundary,resolvent}, we obtain
\begin{align}\label{pay,xi psi L2}
\|\widetilde{u}_{I}\|_{L^2}\leq\nu^{-\frac16}|k|^{-\frac16}\| \widetilde{F}_1\|_{L^2}+\nu^{-\frac12}|k|^{-\frac12}\|\widetilde{F}_2\|_{L^2}.
\end{align}
Due to Plancherel's theorem, we arrive at
\begin{align*}
\|e^{\epsilon (\nu k^2)^{\frac13} t}u_{I}\|_{L^2_tL^2_y}\leq \nu^{-\frac16}|k|^{-\frac16}\|e^{\epsilon (\nu k^2)^{\frac13} t} f_1\|_{L^2_tL^2_y}+\nu^{-\frac12}|k|^{-\frac12}\|e^{\epsilon (\nu k^2)^{\frac13}t} f_2\|_{L^2_tL^2_y}.
\end{align*}
\end{proof}
\begin{proof}[Proof of Proposition \ref{Prop homo pro enha}]
Similar analysis as in the proof of Lemma \ref{Prop: inhomo enhance}, we can reduce the bound of $\|e^{\epsilon(\nu k^2)^{\frac13} t}u\|_{L^\infty_t L^2_y}$ to estimate $|k| \| e^{\epsilon(\nu k^2)^{\frac13} t} u\|^2_{L^2_tL^2_y}$.

We decompose $\om=\om^{(1)}+\om^{(2)}+\om^{(3)}$ as in Proposition \ref{Prop homo pro}, where $\om^{(1)},\, \om^{(2)},\, \om^{(3)}$ satisfy \eqref{equ: omH1}, \eqref{equ: omH2} and \eqref{equ: omH3}, respectively.
Then, we divide the proof into two steps.

\textbf{Step 1.} Estimates of $u^{(j)}$, $j=1,2$.
By Lemma \ref{Prop: inhomo enhance}, we have
\begin{equation}\label{esti:u2}
\begin{aligned}
&\|e^{\epsilon (\nu k^2)^{\frac13}t}u^{(2)}\|_{L^2_tL^2_y}\\
\lesssim &\nu^{-\frac16}|k|^{\frac13}\|e^{\epsilon (\nu k^2)^{\frac13}t}(\nu k^2-(\nu k^2)^{\frac13})\om^{(1)}\|_{L^2_tL^2_y}+\nu^{-\frac12} \|e^{\epsilon (\nu k^2)^{\frac13}t}\nu \pa_y \om^{(1)}\|_{L^2_tL^2_y}\\
\lesssim &\nu^{\frac16}|k|\|e^{\epsilon (\nu k^2)^{\frac13}t}\om^{(1)}\|_{L^2_tL^2_y}+\nu^{\frac12}\|e^{\epsilon (\nu k^2)^{\frac13}t} \pa_y \om^{(1)}\|_{L^2_tL^2_y}.
\end{aligned}
\end{equation}
We use \eqref{esti: om1HL2} and \eqref{payom1HL2L2} to get
\begin{align*}
\nu^{\frac16}|k|\|e^{\epsilon (\nu k^2)^{\frac13}t}\om^{(1)}\|_{L^2_tL^2_y}=& \nu^{\frac16}|k|\|e^{-(1-\epsilon)(\nu k^2)^{\frac13}t}\|_{L^2(0,\infty)} \|\om^{in}\|_{L^2}=\frac{|k|^{\frac23}\|\om^{in}\|_{L^2_y}}{2(1-\epsilon)}
\end{align*}
and
\begin{align*}
\nu^{\frac12} \|e^{\epsilon (\nu k^2)^{\frac13}t}\pa_y \om^{(1)}\|_{L^2_tL^2_y}\leq& \nu^{\frac13}|k|^{-\frac13}\|\pa_y \om^{in}\|^2_{L^2_y}+2\|\om^{in}\|_{L^2_y}.
\end{align*}
Inserting the above two estimates into \eqref{esti:u2}, we deduce
\begin{align*}
\|e^{\epsilon(\nu k^2)^{\frac13}t}u^{(2)}\|_{L^2_tL^2_y}\leq C(\nu^{\frac13}|k|^{-\frac13}\|\pa_y \om^{in}\|_{L^2_y}+\|\om^{in}\|_{L^2_y}).
\end{align*}
Using $u^{(1)}=e^{-(\nu k^2)^{\frac13} t}u^{(0)}$ and \eqref{u1HL2L2}, we get
\begin{align*}
\|e^{\epsilon(\nu k^2)^{\frac13}t}u^{(1)}\|_{L^2_tL^2_y}
=\|e^{-(1-\epsilon)(\nu k^2)^{\frac13}t}u^{(0)}\|_{L^2_tL^2_y}
\leq C |k|^{-\frac12}\|\om^{in}\|_{L^2_y}.
\end{align*}

\textbf{Step 2.} Estimates of $u^{(3)}$. Let
\begin{align*}
&\widetilde{w}(\la,\xi,y)=\int^{+\infty}_0\om^{(3)}(t,\xi,y) e^{-it\la+\epsilon(\nu k^2)^{\frac13}t}dt, \\
&\widetilde{u}(\la,y)=\int^{+\infty}_0u^{(3)} e^{-it\la+\epsilon(\nu k^2)^{\frac13}t}dt,\\
&\widetilde{c}_j(\la)=\int^{+\infty}_0a_j(t) e^{-it\la+\epsilon(\nu k^2)^{\frac13}t} dt, \,\, j=1,2.
\end{align*}
Then, we have
\begin{align*}
&(i\la-\nu (\pa^2_y- k^2)+i k y)\widetilde{ w}(\la, k,y)=0,\\
&\widetilde{c}_l(\la)=-\int^1_{-1} \frac{\sinh  k(1+(-1)^{l+1}y)}{\sinh 2 k} \widetilde{w}(\la, k,y) dy, \quad l=1,2,
\end{align*}
which imply
\begin{align*}
\widetilde{u}=-\widetilde{c}_1(\la)\widetilde{u}_1-\widetilde{c}_2(\la)\widetilde{u}_2.
\end{align*}
By \eqref{esti: psiL2omL2} and \eqref{esti: w1+w2L1} \eqref{claim c1}, \eqref{claim c2}, we deduce
\begin{align*}
\|u\|_{L^2_tL^2_y}\leq & \big\||\widetilde{c}_1(\la)|\|\widetilde{u}_1\|_{L^2_y}\big\|_{L^2_{\la}}+\big\||\widetilde{c}_2(\la)|\|\widetilde{u}_2\|_{L^2_y} \big\|_{L^2_{\la}}\\
\leq & C(\|\widetilde{c}_1\|_{L^2_{\la}}+\|\widetilde{c}_2\|_{L^2_{\la}})(\|\widetilde{w}_1\|_{L^1_y}+\|\widetilde{w}_2\|_{L^1_y})\\
\leq &C|k|^{-\frac12}\|\om^{in}\|_{L^2_y}.
\end{align*}
This inequality yields
\begin{align*}
|k|^{\frac12}\|e^{\epsilon(\nu k^2)^{\frac13}t}u^{(3)}\|_{L^2_tL^2_y}\leq &|k|^{\frac12}\|\widetilde{u}(\la)\|_{L^2_{\la}L^2_y}\leq \|\om^{in}\|_{L^2_y}.
\end{align*}
Then, combining the estimate of $u^{(i)}$ $(1\leq i\leq 3)$, we arrive at
\begin{equation}\label{uL2L2}
\begin{aligned}
&|k|^{\frac12}\|e^{\epsilon (\nu k^2)^{\frac13} t}u\|_{L^2_tL^2_y}\\
\leq &|k|^{\frac12}\|e^{\epsilon (\nu k^2)^{\frac13} t}u^{(1)}\|_{L^2_tL^2_y}+|k|^{\frac12}\|e^{\epsilon (\nu k^2)^{\frac13} t} u^{(2)}\|_{L^2_tL^2_y}+|k|^{\frac12}\|e^{\epsilon (\nu k^2)^{\frac13} t}u^{(3)}\|_{L^2_tL^2_y}\\
\leq &C(\nu^{\frac13}|k|^{-\frac13}\|\pa_y \om^{in}\|_{L^2_y}+ \|\om^{in}\|_{L^2_y}).
\end{aligned}
\end{equation}
This completes the proof of Proposition \ref{Prop homo pro enha}.
\end{proof}
\vspace{.2cm}
\noindent
\textbf{Acknowledgements.}  The authors thank to Prof. Dongyi Wei for many valuable discussions. This project was supported by the National Key Research and Development Program of China (No: 2022YFA1005700). Q. Chen was partially supported by National Natural Science Foundation of China (Grant No: 12471149). Z. Li was partially supported by the Postdoctoral Fellowship Program of CPSF (Grant No: GZC20240123). C. Miao was partially supported by National Natural Science Foundation of China (Grant No: 12371095).
\appendix
\section{Some Basic Estimates}
\begin{lemma}[Wirtinger's Inequality] If $f\in C^1[a,b]$ and $f(a)=f(b)=0$, then
\begin{align*}
\int^b_a|f(x)|^2dx\leq \left(\frac{b-a}{\pi}\right)^2 \int^b_a |f'(x)|^2 dx.
\end{align*}
\end{lemma}
\begin{lemma}\label{lemma sinh}
Let $10\nu\leq | k|<1$. We have
\begin{align}
&\Big\|\frac{\sinh ( k(1+y))}{\sinh 2 k}\Big\|_{L^\infty(-1,1)}\leq C, \quad \Big\|\frac{\sinh ( k(1-y))}{\sinh 2 k}\Big\|_{L^\infty(-1,1)}\leq C,\label{1}\\
& \Big\|\frac{\cosh ( k(1+y))}{\sinh 2 k}\Big\|_{L^2(-1,1)}\leq | k|^{-1},\quad \Big\|\frac{\cosh ( k(1-y))}{\sinh 2 k}\Big\|_{L^2(-1,1)}\leq | k|^{-1}.\label{2}
\end{align}
\end{lemma}
\begin{proof}
Due to
\begin{align*}
\frac{\sinh ( k(1+y))}{\sinh 2 k}=\frac{e^{k(1+y)}-e^{-k(1+y)}}{e^{2k}-e^{-2k}}=\frac{e^{k(y-1)}(1-e^{-2k(y+1)})}{1-e^{-4k}},
\end{align*}
for $10\nu\leq |k|<1$, we obtain
\begin{align*}
\Big\|\frac{\sinh ( k(1+y))}{\sinh 2 k}\Big\|_{L^\infty(-1,1)}\lesssim \left\|\frac{2k(1+y)}{4k}\right\|_{L^\infty(-1,1)}\leq C.
\end{align*}
For $\frac{\cosh ( k(1+y))}{\sinh 2 k}$, thanks to $10\nu\leq |k|<1$, we get
\begin{align*}
\Big\|\frac{\cosh ( k(1+y))}{\sinh 2 k}\Big\|_{L^2(-1,1)} =\frac{\|e^{k(1+y)}+e^{-k(1+y)}\|_{L^2(-1,1)}}{|e^{2k}-e^{-2k}|}\leq \frac{|e^{4k}-e^{-4k|}}{|k||e^{2k}-e^{-2k}|}\leq C|k|^{-1}.
\end{align*}

In a similar way, we can obtain the estimates in \eqref{1} and \eqref{2} including the factors  $\sinh ( k(1-y))$ and $\cosh ( k(1-y))$.
\end{proof}
\begin{lemma}\label{claim}
For $|k|>0$, define
\begin{align*}
\tilde{c}_{l}(\la)=\int^{+\infty}_{0} a_{l}(t) e^{-it\la} dt
\end{align*}
with
\begin{equation*}
\begin{aligned}
a_l(t)=&\int^1_{-1}\frac{\sinh  k (1+(-1)^{l+1}y )}{\sinh 2 k}e^{-(\nu k^2)^{\frac13}t-ikyt} g(k,y) dy,\quad l=1,2.
\end{aligned}
\end{equation*}
If $g\in L^2_y$, then we have
\begin{align}
&\|(1+|\la/ k+1|)\widetilde{c}_1\|^2_{L^2_{\la}}\leq C|k|^{-1}\|g\|^2_{L^2_y},\label{claim c1}\\
&\|(1+|\la/ k-1|)\widetilde{c}_2\|^2_{L^2_{\la}}\leq C|k|^{-1}\|g\|^2_{L^2_y}.\label{claim c2}
\end{align}
\end{lemma}
\begin{proof}
It is clear that
\begin{align*}
\|(1+|\la/ k+1|)\widetilde{c}_1(\la)\|^2_{L^2_{\la}}=|k|\int_{\R}(1+|\la|)^2|\widetilde{c}_1( k(\la-1))|^2 d\la
\end{align*}
and
\begin{align*}
\widetilde{c}_1( k(\la-1))=\int^{+\infty}_0 a_1(t) e^{itk-it\la k}dt.
\end{align*}
Using Plancherel's formula, we deduce that
\begin{align*}
|k|\|\widetilde{c}_1(k(\la-1))\|^2_{L^2_{\la}}\leq \|e^{it k}a_1(t)\|^2_{L^2_t}, \quad
|k|\||\la|\widetilde{c}_1(k(\la-1))\|^2_{L^2_{\la}}\leq |k|^{-2} \|\pa_t (e^{it k}a_1)\|^2_{L^2_t}.
\end{align*}
Thus, \eqref{claim c1} can be reduced to prove
\begin{align}\label{resolu}
\|e^{it k}a_1(t)\|^2_{L^2_t}\leq | k|^{-1}\|g\|^2_{L^2}, \quad \|\pa_t (e^{it k}a_1(t))\|^2_{L^2_t}\leq C|k|\|g\|^2_{L^2}.
\end{align}

Denote
\begin{align}\label{def:b}
b_1(t)=\int^1_{-1}\frac{\sinh  k (y+1)}{\sinh 2 k}e^{-ikyt} g(k,y) dy.
\end{align}
Then, we have
\begin{align}\label{esti: a1L2}
\|e^{i k t} a_1(t)\|^2_{L^2_t}=\|e^{i k t-(\nu  k^2)^{1/3}t} b_1(t)\|^2_{L^2_t}
\end{align}
and
\begin{equation}\label{esti: a1H1}
\begin{aligned}
\|\pa_t (e^{i k t} a_1(t))\|^2_{L^2_t}
=&\left\|e^{i k t-(\nu  k^2)^{1/3}t}\left(\pa_t b_1(t)+i k b_1(t)-(\nu  k^2)^{1/3}b_1(t)\right) \right\|^2_{L^2_t}\\
\leq&\|\pa_t b_1(t)+i kb_1(t)\|^2_{L^2_t}+\|(\nu k^2)^{\frac13}b_1(t)\|^2_{L^2_t}.
\end{aligned}
\end{equation}

It follows from Plancherel's formula that
\begin{align}\label{a01L2L2}
\int_{\R}|b_1(t)|^2 dt=\frac{2\pi}{| k|}\int^1_{-1}\Big|\frac{\sinh  k(y+1)}{\sinh 2 k}g( k,y)\Big|^2 dy\leq \frac{2\pi}{| k|}\|g\|^2_{L^2}.
\end{align}
By \eqref{def:b}, we obtain
\begin{align*}
\pa_t b_1(t)+i k b_1(t)=\int^1_{-1}i k(1-y)\frac{\sinh  k(y+1)}{\sinh 2 k}e^{-it k y} g dy.
\end{align*}
By Plancherel's formula and \eqref{1}, we obtain
\begin{equation}
\begin{aligned}\label{pata1+ikta1L2L2}
\int_{\R}|\pa_t b_1(t)+i k b_1(t)|^2dt=&\frac{2\pi}{| k|}\int^1_{-1}\Big|i k(1-y)\frac{\sinh  k(y+1)}{\sinh 2 k} g \Big|^2dy
\leq 2\pi| k|\|g\|^2_{L^2}.
\end{aligned}
\end{equation}
Putting \eqref{a01L2L2}, \eqref{pata1+ikta1L2L2} into  \eqref{esti: a1L2} and \eqref{esti: a1H1}, we directly obtain \eqref{resolu} and hence completes the proof of \eqref{claim c1}. In a similar way, we can obtain \eqref{claim c2}.
\end{proof}
\begin{lemma}\label{Lemma A1}
Let $(\pa^2_y- k^2)\psi=w$, $\psi(\pm1)=0$. For $0\leq |k|\leq 1$, we have
\begin{align}
&\|\psi''\|^2_{L^2}+2|k|^2\|\psi'\|^2_{L^2}+|k|^4\|\psi\|^2_{L^2}=\|w\|^2_{L^2},\label{psi2L2}\\
&\|\psi'\|^2_{L^2}+\|\psi\|^2_{L^2}\leq C\|w\|^2_{L^1}\leq C\|w\|^2_{L^2},\label{esti: psiL2omL2}\\
&\|\psi'\|_{L^\infty}+\|\psi\|_{L^\infty}\leq C\|w\|_{L^1}\leq C\|w\|_{L^2},\label{esti: psiLinftywL2}\\
&\|\psi'\|_{L^\infty}\leq \|\psi'\|^{1/2}_{L^2}\|w\|^{1/2}_{L^2},\label{esti: psiLinfty inte}
\end{align}
where $\psi'=\pa_y \psi$ and $\psi''=\pa_{yy}\psi$.
\end{lemma}
\begin{proof}
Taking the $L^2$ inner produce with $w$, we obtain
\begin{align*}
\|\psi''\|^2_{L^2}+2|k|^2\|\psi'\|^2_{L^2}+|k|^2\|\psi\|^2_{L^2}=\lan (\pa^2_y- k^2)\psi,(\pa^2_y- k^2)\psi\ran=\lan w,w\ran=\|w\|^2_{L^2}.
\end{align*}

For \eqref{esti: psiL2omL2}, using Wirtinger's inequality, we have
\begin{align*}
\|\psi\|_{L^2}\leq \frac{2}{\pi}\|\psi'\|_{L^2}.
\end{align*}
Then, it suffices to prove
$$\|\pa_y\psi\|_{L^2}\leq C\|w\|_{L^1}.$$
Due to
\begin{align}\label{psiy}
|\psi(y)|=\left|\int^y_{-1}\psi'(\tilde{y}) d\tilde{y}\right|\leq (y+1)^{\frac12}\|\psi'\|_{L^2},
\end{align}
 we have
\begin{align*}
\|\psi'\|^2_{L^2}+| k|^2\|\psi\|^2_{L^2}=\lan-w,\psi\ran\leq \|w\|_{L^1}\|\psi\|_{L^\infty}\leq 2\|w\|_{L^1}\|\psi'\|_{L^2},
\end{align*}
which implies
\begin{align*}
\|\psi'\|_{L^2}\leq 2\|w\|_{L^1}.
\end{align*}

For \eqref{esti: psiLinftywL2}, by interpolation and \eqref{esti: psiL2omL2}, we infer that
\begin{align*}
\|\psi\|_{L^\infty}\leq \|\psi'\|^{1/2}_{L^2}\|\psi\|^{1/2}_{L^2}\leq \|w\|_{L^1}.
\end{align*}
For $y\in [0,1]$, there exists $y_1\in (y-1, y)\cap(-1,1)$ so that $|\psi'(y_1)|\leq \|\psi'\|_{L^2}$. Then we have
\begin{align*}
|\psi'(y)|\leq& |\psi'(y_1)|+\int^y_{y_1}|\psi''(z)|dz\\
\leq& \|\psi'\|_{L^2}+\int^{y}_{y_1}| k^2\psi(z)+w(z)|dz\\
\leq& \|\psi'\|_{L^2}+| k|^2\|\psi\|_{L^\infty}+\|w\|_{L^1}\leq C\|w\|_{L^1}.
\end{align*}
This inequality yields $\|\psi'\|_{L^\infty}\leq C\|w\|_{L^1}.$

By making use of interpolation and \eqref{psi2L2}, we have
\begin{align*}
\|\psi'\|_{L^\infty}\leq \|\psi'\|^{1/2}_{L^2}\|\psi''\|^{1/2}_{L^2}\leq \|\psi'\|^{1/2}_{L^2}\|w\|^{1/2}_{L^2}.
\end{align*}
\end{proof}
\section{Airy Function}\label{sec: Airy Function}
\setcounter{equation}{0}
\renewcommand\theequation{B.\arabic{equation}}
Let $Ai(z)$ be the classical Airy function which satisfies
\begin{align*}
\pa^2_z Ai(z)-zAi(z)=0.
\end{align*}
As presented in \cite{Romanov,Wasow,Chen-Li-Wei-Zhang}, for $|\mathrm{arg} z|\leq \pi-\varepsilon$, $\varepsilon>0$, $Ai(z)$ has the following asymptotic formula
\begin{align*}
Ai(z)=\frac{1}{2\sqrt{\pi}}z^{-\frac14} e^{-\frac23 z^{\frac32}}(1+R(z)),\quad R(z)=O(z^{-\frac32}).
\end{align*}
Define
\begin{align*}
A_0(z)=\int^{+\infty}_{e^{i\pi/6}z}Ai(t)dt=e^{i\pi/6}\int^{+\infty}_{z}Ai(e^{i\pi/6 }t)dt.
\end{align*}
Here, we show some properties for $A_0(z)$. 
\begin{lemma}{\rm\cite{Romanov}}\label{A1}
There holds that
\begin{itemize}
\item[1.] There exists $\delta_0>0$ such that $A_0(z)$ has no zeros in the half plane $\mathrm{Im} z\leq \delta_0$.
\item[2.] Let $a(\delta)=\sup\Big\{\mathrm{Re}\Big(\frac{A'_0(z)}{A_0(z)}\Big): \mathrm{Im} z\leq \delta\Big\}$. There exists $\delta_0>0$ so that $a(\delta)\in C([0,\delta_0])$ and
    \begin{align*}
    a(0)=-0.4843\cdots.
    \end{align*}
\item[3.] For $|\mathrm{arg}(ze^{\frac{i\pi}{6}})|\leq \pi-\varepsilon$, $\varepsilon>0$, we have the asymptotic formula
\begin{align*}
\frac{A'_0(z)}{A_0(z)}-e^{i\pi/6}(ze^{i\pi/6})^{\frac12}+O(z^{-1}).
\end{align*}
\end{itemize}
\end{lemma}
\begin{lemma}\rm{\cite{Chen-Li-Wei-Zhang}}\label{A2}
Let $\delta_0$ be as in Lemma \ref{A1}. there exists $c>0$ so that for $\mathrm{Im} z\leq \delta_0$,
\begin{align*}
\Big|\frac{A'_0(z)}{A_0(z)}\Big|\leq 1+|z|^{\frac12},\quad Re\frac{A'_0(z)}{A_0(z)}\leq -c(1+|z|^{\frac12}).
\end{align*}
\end{lemma}
Define
\begin{align*}
\eta(z,x)=\frac{A_0(z+x)}{A_0(z)}=\exp\Big(\int^x_0 \frac{A'_0(z+t)}{A_0(z+t)}dt\Big),
\end{align*}
we have
\begin{lemma}\rm{\cite{Chen-Li-Wei-Zhang}}\label{A3}
For $0<\sigma<-a(0)$, there exists $\delta_1>0$ so that for $\mathrm{Im} z\leq \delta_1$ and $x\geq 0$,
\begin{align*}
|\eta(z,x)|\leq e^{-a_{\sigma}x},
\end{align*}
with $a_{\sigma}=-a(0)-\sigma$.
\end{lemma}

\end{document}